\documentclass{siamart220329}
\usepackage[utf8]{inputenc}
\usepackage{amsmath,amssymb,a4wide,bbm,esdiff,mathtools}
\usepackage{xcolor}
\usepackage{diagbox}
\usepackage{dashbox}
\usepackage{lipsum}
\usepackage{enumitem}
\usepackage{algorithm}  
\usepackage{algpseudocode}

\usepackage{forest} 
\usepackage{sidecap}
\usepackage[caption=false]{subfig}

\newenvironment{butchertableau}[2][1.25]{\def\arraystretch{#1}\array{#2}}{\endarray}

\newcommand{\myremarkend}{$\spadesuit$} 
\newtheorem{remark}{Remark}
\newtheorem{algodef}{Algorithm}

\newcommand{\rk}{RK}
\newcommand{\erk}{explicit RK}
\newcommand{\Erk}{Explicit RK}

\newcommand{\order}{{\mathcal O}}
\newcommand{\numc}{n_c}
\newcommand{\R}{{\mathbb R}}
\newcommand{\dt}{\Delta t}

\DeclareMathOperator{\Span}{span}
\DeclareMathOperator{\Rank}{rank}

\newcommand{\Kspace}{{\mathcal K}}
\newcommand{\Kop}{{\mathcal K}}

\DeclareMathOperator{\Col}{col}

\newcommand{\mU}{\hat{A}}

\newcommand{\Vtilde}{\widetilde{V}}
\newcommand{\Stilde}{\widetilde{S}}

\newcommand{\Bmat}{B}

\newcommand{\Xmat}{X}

\newcommand{\xsum}{\bar{x}}
\newcommand{\jind}{k}

\newcommand{\mq}{q}
\newcommand{\power}[2]{ {#1}^{\,#2}}
\newcommand{\mA}{\mathcal{A}}
\newcommand{\mD}{\mathcal{D}}

\newcommand{\rkone}[1]{\widehat{#1}}
\newcommand{\rktwo}[1]{\breve{#1}}

\newcommand{\bracChoice}{bmatrix}
\newcommand{\lbrac}{\left[} 
\newcommand{\rbrac}{\right]}

\newif\ifarxiv
\arxivfalse

\headers{Explicit Runge--Kutta Methods that Alleviate Order Reduction}{A. Biswas, D.I. Ketcheson, S. Roberts, B. Seibold, D. Shirokoff}

\title{Explicit Runge--Kutta Methods that Alleviate Order Reduction%
\thanks{Submitted to the editors DATE.%
\funding{This work was performed under the auspices of the U.S. Department of Energy by Lawrence Livermore National Laboratory under Contract DE-AC52-07NA27344 and was supported by the LLNL-LDRD Program under Project No. 23-ERD-048. LLNL-JRNL-855353.
This material is based upon work supported by the National Science Foundation under Grant No.~DMS--2309728 (Seibold) and DMS--2309727 (Shirokoff). Any opinions, findings, and conclusions or recommendations expressed in this material are those of the authors and do not necessarily reflect the views of the National Science Foundation.}%
}%
}
\author{Abhijit Biswas\thanks{Computer, Electrical, and Mathematical Sciences \& Engineering Division, King Abdullah University of Science and Technology, Thuwal 23955, Saudi Arabia, (\email{abhijit.biswas@kaust.edu.sa}, \email{david.ketcheson@kaust.edu.sa}).}
\and David I. Ketcheson\footnotemark[2]
\and Steven Roberts\thanks{Center for Applied Scientific Computing, Lawrence Livermore National Laboratory, Livermore, California 94550, USA (\email{roberts115@llnl.gov})}
\and Benjamin Seibold\thanks{Department of Mathematics, Temple University, Philadelphia, PA 19122, (\email{seibold@temple.edu}).}
\and David Shirokoff\thanks{Department of Mathematical Sciences, New Jersey Institute of Technology, Newark, NJ 07102, (\email{david.g.shirokoff@njit.edu}).}
}

\begin{document}
\forestset{
  */.style={
    delay+={append={[]},}
  },
  rooted tree/.style={
    for tree={
      grow'=90,
      parent anchor=center,
      child anchor=center,
      s sep=2.5pt,
      if level=0{
        baseline
      }{},
      delay={
        if content={*}{
          content=,
          append={[]}
        }{}
      }
    },
    before typesetting nodes={
      for tree={
        circle,
        fill,
        minimum width=3pt,
        inner sep=0pt,
        child anchor=center,
      },
    },
    before computing xy={
      for tree={
        l=5pt,
      }
    }
  },
  big rooted tree/.style={
    for tree={
      grow'=90,
      parent anchor=center,
      child anchor=center,
      s sep=2.5pt,
      if level=0{
        baseline
      }{},
      delay={
        if content={*}{
          content=,
          append={[]}
        }{}
      }
    },
    edge={
        semithick,
        -Latex
    },
    before typesetting nodes={
      for tree={
        circle,
        fill,
        minimum width=5pt,
        inner sep=0pt,
        child anchor=center,
      },
    },
    before computing xy={
      for tree={
        l=15pt,
      }
    }
  },
  my node/.style={
    circle, fill=#1,
  },
  gray node/.style={
    my node=gray, inner sep=#1mm,
  },
}
\DeclareDocumentCommand\rt{o}{\Forest{rooted tree [#1]}}
\DeclareDocumentCommand\bigrt{o}{\Forest{big rooted tree [#1]}}

\maketitle

\begin{abstract}
Explicit Runge--Kutta (\rk{}) methods are susceptible to a reduction
in the observed order of convergence when applied to initial-boundary
value problem with time-dependent boundary conditions. 
We study conditions on \erk{} methods that guarantee
high-order convergence for linear problems; we refer to these conditions as weak stage order conditions.
We prove a general relationship between the method's order, weak stage order, and number of stages.  
We derive \erk{} methods with high weak stage order and demonstrate, through numerical tests, that they avoid the order reduction phenomenon up to any order for linear problems and up to order three for nonlinear problems.
\end{abstract}



\section{Introduction}
When a Runge--Kutta (\rk{}) method is applied to a stiff problem or an initial-boundary value problem (IBVP) with time-dependent boundary conditions, the observed rate of convergence of the solution is often lower than what is predicted by classical analysis.  Instead, the rate of convergence may be given by what is known as the \emph{stage order} of the method.  For many problems, such as hyperbolic PDEs, \erk{} methods are preferred. However, \erk{} methods have stage order equal to (at most) one, so order reduction poses an important challenge. A number of remedies have been provided in the literature, generally taking the form of modified boundary conditions.

Herein we provide an approach that avoids order reduction while imposing boundary conditions in the standard way. It differs from traditional approaches only in that the coefficients of the \rk{} method must be chosen to satisfy a set of additional \emph{weak stage order} conditions. For linear IBVPs, we show that such methods can be designed so as to achieve arbitrarily high order convergence. For nonlinear IBVPs, these methods can attain up to third order accuracy.

\subsection{Previous work on avoiding order reduction for hyperbolic PDEs}
In the context of hyperbolic PDEs, order reduction was studied in \cite{sanz1986convergence,sanz1989stability}.  A simple possibility for avoiding order reduction was noted there, in which one rewrites the original problem as an inhomogeneous PDE with homogeneous boundary conditions.  In \cite{carpenter1995theoretical}, the boundary condition function is differentiated and the resulting equation integrated with the \rk{} method; this approach was shown to avoid order reduction for linear problems, and to provide up to third order convergence for nonlinear problems.  A generalization was later given to avoid order reduction in general for nonlinear problems \cite{abarbanel1996removal}.  These ideas were further developed in \cite{pathria1997correct}, wherein it was suggested to analytically derive formulas for appropriate boundary integration errors rather than computing them with the \rk{} method. Error estimates demonstrating that modified boundary conditions avoid order reduction were also provided in \cite{Alonso2002, AlonsoCano2004}.

However, all of the foregoing modified boundary techniques have a number of drawbacks: they require the availability of derivatives of the boundary conditions; they can reduce the stable time step size; they introduce additional numerical parameters that need to be chosen on a case-by-case basis; and they require a somewhat complicated and intrusive modification of a typical solver code.  Most importantly, they have only been studied and applied in one spatial dimension; their application in higher dimensions is not immediate or straightforward \cite{pathria1997correct}.

\subsection{Previous work on avoiding order reduction through conditions on \rk{} coefficients}
The conditions we refer to here as \emph{weak stage order} (WSO) have been derived previously in various forms under different names.  Scholz \cite{scholz1989order} derived the WSO conditions for singly diagonally implicit \rk{} (SDIRK) schemes up to fourth order\footnote{Formally, Scholz examined the closely related ROW methods; however, when applied to linear problems the conditions are essentially equivalent to those for SDIRK methods.}; he also constructed and tested new schemes, and established order barriers up to order four.  Ostermann and Roche \cite{OstermannRoche1992} later sharpened the work of Scholz by explaining fractional orders of convergence and providing general conditions to avoid order reduction for implicit \rk{} methods applied to linear PDEs. In \cite{OstermannRoche1993}, they proposed closely-related \textit{parabolic order conditions} for Rosenbrock methods.

Subsequent work by Skvortsov provided analysis, examples, and methods for both diagonally implicit \rk{} (DIRK) and \erk{} methods having WSO two or three \cite{skvortsov2006diagonally,skvortsov2010model,skvortsov2017avoid}. In the last of those works, essentially the same concept studied here is referred to as \emph{pseudostage order}. Rang also independently derived conditions equivalent to WSO along with developing and testing new schemes satisfying the conditions, and referring to them as $B_{\rm PR-consistency}$ (\cite{Rang2016}; see equations (12)--(13) therein).  More recent works by some of the present authors and collaborators
took a geometric approach by interpreting the WSO conditions in terms of orthogonal vector spaces \cite{ketcheson2020dirk,rosales2017spatial}. More recently, Rosenbrock methods that avoid order reduction for differential algebraic equations up to order five were also investigated in \cite{Steinebach2023}.

Whereas the previous works on RK methods focused on specific conditions and methods of low to moderate order, here we provide a general theory for explicit schemes that extends to methods of arbitrary order.  The work here can be viewed as an extension of \cite{biswas2022,biswas2023design}, which provided theoretical analysis and practical development of high WSO implicit and DIRK schemes, to \erk{} methods.

\subsection{Weak stage order}
An $s$-stage \rk{} method \cite[pg. 134]{hairer1993solving} is defined by vectors $b, c\in \R^s$, and
matrix $A \in \R^{s\times s}$. We assume throughout this work that $c_i = \sum_j a_{ij}$, so that a method is fully specified by $A$ and $b$.
Given a \rk{} method with coefficients $(A,b)$, we define the spaces
\begin{align}
\label{Y}
Y &:= \Span\left\{b, A^T b, \dots, (A^T)^{s-1} b\right\} \, , \quad \text{and} \\
\label{Km}
K_m &:= \left\{ \tau^{(1)}, A\tau^{(1)},\dots, A^{s-1}\tau^{(1)}, 
            \tau^{(2)}, A\tau^{(2)},\dots, A^{s-1}\tau^{(2)},
            \dots, A^{s-1}\tau^{(m)} \right\} \, ,
\end{align}
with the stage order residuals $\tau^{(k)} := Ac^{k-1} - \frac{1}{k} c^k$.
Note that by construction, $Y$ is the smallest $A^T$-invariant subspace containing $b$, while $K_m$ is the smallest $A$-invariant subspace containing $\tau^{(1)}, \ldots, \tau^{(m)}$.

We define the weak stage order (WSO) of the method (see \cite{ketcheson2020dirk,biswas2022}) as the largest positive integer $q$ such that
$K_q$ is orthogonal to $Y$.
By \cite[Theorem 2]{ketcheson2020dirk}, WSO $q$ is also equivalent to
\begin{equation} \label{eq:WSO_rational}
    0 = b^T (I - z A)^{-1} \tau^{(k)} \,, \qquad \forall z \in \mathbb{C} \,, \quad 1 \leq k \leq q \,.
\end{equation}
If \eqref{eq:WSO_rational} holds for all $k \geq 1$, then the scheme has WSO $q = \infty$; for example the explicit Euler method has $A=0$ and thus has WSO $q=\infty$.
Therefore in the rest of this work
we focus on methods with $p>1$.
Any \erk{} method (satisfying the stage
consistency condition as we assume) has $\tau^{(1)}=0$,
hence $\dim(K_1)=0$.

For many \rk{} schemes, satisfying the WSO conditions turns out to be sufficient for avoiding order reduction entirely on linear problems \cite{sanz1986convergence,OstermannRoche1992,Rang2016}.
A major advantage over the (stronger) stage order conditions \cite{hairer1996solving} is that explicit methods can have arbitrarily high WSO, as show herein.
In the context of PDE approximations, order reduction for a method with order $p$ and stage order $q$ typically reduces the convergence rate to $\min(p,q+1)$ (see e.g.~\cite[Section~II.2.3]{hundsdorfer2003numerical}).  In our tests in \cref{sec:tests} we will see similarly that for a method with order $p$ and WSO $q$, the observed rate of convergence (for linear problems with time-dependent boundary conditions) is $\min(p,q+1)$.  It is therefore natural to devise methods with $q=p-1$. For nonlinear PDEs, numerical evidence suggests WSO conditions partially mitigate order reduction with reduction to order three at worst.

\subsection{Contributions and outline of this work}
This work's main contributions are:
\begin{itemize}
    \item theoretical results providing an understanding of and
        enabling the construction of high WSO \erk{} methods (Thms.~\ref{thm:qm1}--\ref{thm:qm3});
    \item characterization of the linear and nonlinear stability properties of such methods (\cref{sec:structure})
    \item construction of optimally efficient (in terms of number of stages) high WSO \erk{} methods up to order 5 and WSO 5 (\cref{sec:construction});
    \item construction of methods with arbitrarily high order and WSO (\cref{sec:high}); and
    \item demonstration of the methods' effectiveness in preventing or mitigating order reduction on a range of problems (\cref{sec:tests}).
\end{itemize}
This manuscript proceeds as follows. The WSO conditions impose significant restrictions on the coefficients and stability properties of an \erk{} method, as we discuss in \cref{sec:structure}. In \cref{sec:construction} we construct methods with high WSO in various ways. First, in \cref{subsec:conERKslow}, we algebraically derive families of methods with WSO two through five. In \cref{sec:high} we show (by construction) that there exist explicit methods with arbitrarily high WSO. In \cref{sec:implementation} we discuss the efficient implementation of the schemes, and in \cref{sec:tests} selected schemes are tested on a range of numerical examples.  
Technical proofs of the main results are given in \cref{Sec:PfsMainThm}.
Conclusions are discussed in \cref{sec:conclusion}.

Coefficients of all new methods proposed in this work can be found in the supplementary materials and in \cite{ERK_High_WSO_Code}.

\section{Structural restrictions} \label{sec:structure}
In this section we state the main theoretical results of this study; short proofs are included here, while longer technical proofs are deferred to \cref{Sec:PfsMainThm}.

\subsection{Main results}\label{subsec:mainresults}
This section introduces bounds on the weak stage order of \erk{} schemes, and characterizes properties of schemes with a minimal number of stages. We first introduce two lemmas which play a key role in the work, and then state the main results. Let 
\begin{equation*}
R(z) := 1 + \sum_{j=1}^s z^j b^T A^{j-1} e
\end{equation*}
denote the stability function of the method (here $e\in \mathbb{R}^s$ has all entries equal to one).
\begin{lemma} \label{dim-lemma}
Let an \erk{} method with coefficients $(A, b)$, stability function $R(z)$, order of accuracy $p$, and weak stage order $q$ be given and let $Y, K_m$ be defined by \eqref{Y}--\eqref{Km}.
Then $Y$ and $K_q$ are orthogonal, and
\begin{align} \label{dimcond}
 p \le \deg(R(z)) \le   \dim(Y) \le s - \dim(K_q) \, .
\end{align}
\end{lemma}
\begin{proof}
The rightmost inequality is \cite[Lemma~2.1]{biswas2022} and the leftmost inequality holds because for a method of order $p$: $R(z) - e^z = \order(z^{p+1})$.
The middle inequality follows from explicit formulas for the stability polynomial via \cite[Theorem~3.2]{biswas2022} since $D(z) = 1$ for an explicit scheme.
\end{proof}

Any explicit scheme with $p>1$ has $\tau^{(2)}\ne 0$, so $\dim(K_2)\ge 1$.  Nearly all \erk{} methods used in practice have $\deg(R(z))=s$, so Lemma \ref{dim-lemma} implies that they have $q = 1$.
In fact we see immediately from this Lemma that any method with $p>1$ and $q>1$ must have $\deg(R(z))<s$, which implies that at least one of the entries on the first subdiagonal of $A$ must vanish.
Thus, \erk{} methods with high weak stage order require a structure that differs from that of most existing methods. 

We also see from \eqref{dimcond} that, in order to achieve high order $p$ and WSO $q$ with as few stages as possible, it is desirable to minimize the dimension of the space $K_q$.
It turns out that the best one can do is $\dim(K_q)=q-1$.

\begin{lemma}\label{Lem:LBdimK}
    Let an \erk{} method be given with order $p$, weak stage order $q$, and $\numc$ distinct abscissas.  
    If $p>1$ then $q < \numc$. If $q < \numc$ then $\dim K_q \geq q - 1$.    
\end{lemma}
\begin{proof}
    For brevity the proof of Lemma~\ref{Lem:LBdimK} is deferred to \cref{Sec:PfsMainThm}.
\end{proof}
For methods of order two or higher, the sum of the order and weak stage order is 
bounded in terms of the number of stages.
\begin{theorem}[Bound on the number of stages]\label{thm:spq}
    Let an \erk{} method with order $p \geq 2$, weak stage order $q$, $s$ stages and $\numc$ distinct abscissa be given. Then $q < \numc$ and moreover
        \begin{align} \label{pqs-bound}
            p + q \le s+1 \, .
        \end{align}    
        If \eqref{pqs-bound} holds with equality then $\dim(K_q) = q-1$ and $\dim(Y)=p$.
\end{theorem}
\begin{proof}
    Combining the lower bound $\dim K_q \geq q - 1$ in Lemma~\ref{Lem:LBdimK} with \eqref{dimcond} yields \eqref{pqs-bound}.
\end{proof}

\begin{remark}
    The bound \eqref{pqs-bound} holds even if $p$ is taken to be the order of the method with respect only to linear, constant-coefficient problems (i.e., if the method is only required to satisfy the tall-tree order conditions). \myremarkend
\end{remark}

The next two theorems provide a way to construct schemes with high WSO $q$, based on implications of Theorem~\ref{thm:spq}.  In them, we refer to the following block structure of $A$ and $c$:
\begin{align}\label{Eq:BlockA}
    A = \begin{\bracChoice}
       0 & 0 & 0 \\
       A_{21} & A_{22} & 0 \\
       A_{31} & A_{32} & A_{33} 
    \end{\bracChoice} , \quad
    c = \begin{\bracChoice}
        0 \\
        c_U \\
        c_L
    \end{\bracChoice} , \quad \textrm{where} \; 
    A_{22} \in \mathbb{R}^{(q-1) \times (q-1)} \, , \; A_{33} \in \mathbb{R}^{(s-q) \times (s-q)} \, . 
\end{align}
In \eqref{Eq:BlockA}, $c_U \in \mathbb{R}^{q-1}$ and $c_L \in \mathbb{R}^{s-q}$. The following Vandermonde-like matrices, defined in terms of $c$, will also play a role
\begin{alignat*}{2}
    V_{U} &= \lbrac\begin{array}{c|c|c|c} 
                    c_U & c_U^2 & \cdots & c_U^{q-1} 
                \end{array}\rbrac \, , \quad & \quad 
    W_{U} &= \lbrac\begin{array}{c|c|c|c} 
                    \tfrac{1}{2} c_U^2 & \tfrac{1}{3} c_U^3 & \cdots & \tfrac{1}{q} c_U^{q} 
                \end{array}\rbrac \, ,   \\
    V_{L} &= \lbrac\begin{array}{c|c|c|c} 
                    c_{L} & c_L^2 & \cdots & c_L^{q-1} 
                \end{array}\rbrac \, , \quad & \quad 
    W_{L} &= \lbrac\begin{array}{c|c|c|c} 
                    \tfrac{1}{2} c_L^2 & \tfrac{1}{3} c_L^3 & \cdots & \tfrac{1}{q} c_L^{q}
                \end{array}\rbrac \, .
\end{alignat*}
\begin{theorem}[Necessary Conditions for WSO $q$]\label{thm:qm1}
Any irreducible \erk{} method with WSO $q$ and $\dim(K_q)=q-1$ must have the following properties:
\begin{enumerate}
    \item[(a)] The first $q+1$ abscissas $c_1, c_2, \dots, c_{q+1}$ are distinct.
    \item[(b)] A unique $L \in \mathbb{R}^{(s-q) \times (q-1)}$ exists that simultaneously satisfies the two Sylvester equations
        \begin{align} \label{Eq:FirstSylvester}
            A_{33} L - L W_{U} V_{U}^{-1} &= \left( A_{33} V_{L} - W_{L} \right) V_{U}^{-1} \, , \\ \label{Eq:SecondSylvester}
            L A_{22} - A_{33} L &= A_{32} \, .
        \end{align}        
    \item[(c)] The vector of weights $b$ admits the following form in terms of the matrix $L$ in part (b):
        \begin{align}\label{Eq:bvector}
                b = \begin{\bracChoice}
                    1 &  0 \\
                    0 & -L^T \\
                    0 &  I 
                \end{\bracChoice} \beta \, , \quad \textrm{for some} \quad \beta \in \mathbb{R}^{s-q + 1} \, .
        \end{align}
    \item[(d)] $a_{q+1,q}=0$.
\end{enumerate}
\end{theorem}
In addition to the necessary conditions in Theorem~\ref{thm:qm1} we have the following related (but separate) sufficient conditions. 
\begin{theorem}[Sufficient Conditions for WSO of at least $q$]\label{thm:qm2} 
    An \erk{} scheme $(A,b)$, with a pair $(L, \beta)$ satisfying the conditions stated in Theorem~\ref{thm:qm1}(a--d) has WSO at least $q$.
\end{theorem}
The proofs of Theorems~\ref{thm:qm1} and \ref{thm:qm2} are technical and deferred to \cref{Sec:PfsMainThm}.  These theorems provide a systematic way to construct schemes with WSO of at least $q$.  

The fact that there is a \emph{gap} between the necessary and sufficient conditions in Theorem~\ref{thm:qm1} is removed when the stage bound \eqref{pqs-bound} is satisfied with equality.

\begin{theorem}[Schemes with minimum number of stages]\label{thm:qm3}
An \erk{} method with $s$ stages and order $p \geq 2$ has WSO $s + 1 - p$ if and only if the conditions stated in Theorem~\ref{thm:qm1}(a--d) hold with $q = s - p + 1$ therein. Furthermore, such schemes are irreducible and the stability function of the method is the $p$-th partial sum of the exponential, i.e., $R(z) = \displaystyle\sum_{j=0}^p \frac{z^j}{j!}$.
\end{theorem}
\begin{proof}
If an $s$-stage scheme with order $p \geq 2$ satisfies the algebraic conditions from Theorem~\ref{thm:qm1}(a--d) for $q = s - p + 1$, then by Theorem~\ref{thm:qm2} it has WSO $q' \geq q$. According to inequality \eqref{pqs-bound}, $q' \leq s - p + 1 = q$, whence $q' = q$. 
\quad\quad
Conversely, if the scheme has WSO $q = s + 1 - p$, then Theorem~\ref{thm:spq} implies $\dim K_q = q - 1$ and $\dim Y = p$. Moreover, according to inequality \eqref{pqs-bound}, no scheme can exist with the same values $p$ and $q$ but smaller $s$, so the scheme must be irreducible. Hence the scheme satisfies the hypothesis of Theorem~\ref{thm:qm1}. Inequality \eqref{dimcond} then implies $\deg(R(z)) = p$, which restricts $R(z)$ to be the $p$-th partial sum of the exponential.
\end{proof}

\subsection{Strong stability preservation}\label{subsec:ssp}
\Erk{} methods are widely used when solving nonlinear hyperbolic PDEs; commonly used methods in this context are \emph{strong stability preserving} (SSP) \rk{} methods. Unfortunately, we here show that the SSP property (for nonlinear problems) is incompatible with high weak stage order.
In the proof of the following theorem, we make use of the fact that if two vectors are entry-wise non-negative and their inner product vanishes, then they have disjoint support.  In particular, this implies
\begin{alignat*}{4}
    \textrm{(Non-negativity 1)}&& \quad\quad\quad\quad  & \textrm{If} \; u \geq 0 \; \textrm{and} \; u^T e = 0 \;   &\quad&\Longrightarrow &\quad       u &= 0\;, \\
    \textrm{(Non-negativity 2)}&& \quad\quad\quad\quad  & \textrm{If} \; u \geq 0 \; \textrm{and} \; u^T c^2 = 0 \; &\quad&\Longrightarrow &\quad  u^T c &= 0 \;.
\end{alignat*}

\begin{theorem}[Incompatibility of non-negative coefficients and WSO]\label{thm:positive_ceofficients}
    An \erk{} \; method with non-negative coefficients $A \geq 0$, $b \geq 0$, satisfies one of the following:
    \begin{enumerate}
        \item [(i)] WSO $q = 1$; or
        \item [(ii)] Order $p \leq 1$ and WSO $q = \infty$.       
    \end{enumerate}        
\end{theorem}
\begin{proof}
    Let an \erk{} method be given with $A \geq 0, b \geq 0$ and assume that the method has WSO $q \geq 2$. We show that $p \leq 1$ and $q = \infty$. We prove the following first. Let $k \geq 2$: 
    \begin{align}\label{Eq:InductionStepVer2}
        \textrm{If} \quad \ b^T A^k = 0 \; \quad \textrm{then} \quad  b^T A^{k-1} = 0\;.
    \end{align}
    Starting with the WSO 2 condition $b^T A^{k-2} \left( A^2 e - \frac{1}{2} c^2 \right) = 0$ and using $b^T A^k = 0$ we have    
    \begin{alignat*}{5}
                       &   &   \quad b^T A^{k-2} c^2  &= 0  \, , &      &\quad \Longrightarrow    &       \quad b^T A^{k-2} c &= 0  \;   &   &\quad \textrm{(Non-negativity 2 with } u = (A^T)^{k-2} b\textrm{)} \, ,  \\ 
        \Longrightarrow&   &   \quad    b^T A^{k-1} e &= 0  \, , &      &\quad \Longrightarrow    &         \quad b^T A^{k-1} &= 0 \;    &    &\quad \textrm{(Non-negativity 1)} \;. 
    \end{alignat*}
    Now, since $b^T A^{s} = 0$ (trivially for an explicit scheme), induction via \eqref{Eq:InductionStepVer2} implies $b^T A = 0$ and thus $b^Tc=0$, so the method has order at most $p = 1$.  Furthermore, if $b^T \tau^{(2)} = 0$ then $b^T \tau^{(j)} = 0$ for all $j \geq 1$ (since $b^T c^j = 0$ for all $j$ and $b^T A = 0$). Hence, $q = \infty$. 
\end{proof}

\begin{corollary}
    Any \erk{} method with order $p \geq 2$ and WSO $q \geq 2$ has an SSP coefficient of zero. 
\end{corollary}
\begin{proof}
    A necessary condition for a non-zero SSP coefficient of an irreducible \rk{} method is $A \geq 0$ and $b > 0$ (see \cite[Theorem~4.2]{kraaijevanger1991contractivity}).
    But by Theorem \ref{thm:positive_ceofficients}, this is only possible for an \erk{} scheme if $p \leq 1$ or $q = 1$.
\end{proof}

High WSO \erk{} methods can still be conditionally SSP for \emph{linear, constant-coefficient} problems. The SSP step size in this setting depends only on the stability function of the method, and for each method discussed below, we include the linear SSP coefficient.

\section{Construction of methods} \label{sec:construction}
In this section we construct practical \erk{} schemes with high WSO (i.e. $q>1$). In \cref{subsec:conERKslow} we construct schemes with a minimum number of stages (i.e., $s = p+q - 1$) and optimized error constants up to orders $p, q \leq 5$. In \cref{sec:high} we construct parallel iterated \rk{} methods with arbitrarily high order and weak stage order. 

We denote a method with $s$ stages, order $p$, and WSO $q$ via the triplet $(s,p,q)$. We also restrict our focus to methods where $q = p$ or $q = p - 1$ since \erk{} schemes with high WSO typically achieve an error of $\min(p, q+1)$ in (linear) PDE problems or $\min(p, q)$ in ODE problems. 

\subsection{Construction of methods with the minimum number of stages}\label{subsec:conERKslow} 
Theorem~\ref{thm:qm3} suggests a way to  construct $(s,p,q)$ schemes that achieve equality in \eqref{pqs-bound}. 
First, we simultaneously solve the
quadrature (or bushy-tree) conditions
\begin{equation}\label{eq:quadrature-conditions}
    b^T c^{j-1}  = \tfrac{1}{j}
	\quad\text{for}\quad
	1 \leq j \leq p \,
\end{equation}
together with equations \eqref{Eq:FirstSylvester}--\eqref{Eq:bvector}, which enforce the
the conditions for WSO $q$.

Although the equations \eqref{Eq:FirstSylvester}--\eqref{Eq:bvector} and \eqref{eq:quadrature-conditions} are nonlinear, they enable a solution through a sequence of two linear systems if we fix in advance $A_{22}, A_{33}, c$. This leads to the following algorithm:
\begin{algodef}[Parametric Solution of \eqref{Eq:FirstSylvester}--\eqref{Eq:bvector} \& \eqref{eq:quadrature-conditions}]\label{alg:cap}
    \phantom{New}          
    \begin{enumerate}[nosep] 
        \item[]\hspace{-0.5cm}{\bf \textrm{Inputs:}} $(A_{22}, A_{33}, c) \in \mathbb{R}^{q-1 \times q-1} \times 
    \mathbb{R}^{p-1 \times p-1} \times \mathbb{R}^s$.
        \item[]\hspace{-0.5cm}{\bf \textrm{Constraints:}} $c_1 = 0, c_2, \ldots, c_{q+1}$ distinct, $q \geq p -1$, $p+q=s +1$.        
        \item[]\hspace{-0.5cm}{\bf \textrm{Outputs:}} $(A_{32}, A_{21}, A_{32}, b)
        \in \mathbb{R}^{p-1 \times q-1} \times \mathbb{R}^{q-1} \times \mathbb{R}^{p-1} \times \mathbb{R}^s$.
    \end{enumerate}        
    \begin{enumerate}[nosep] 
        \item[{\bf 1.}] Solve \eqref{Eq:FirstSylvester} to determine $L \in \mathbb{R}^{p-1 \times q-1}$. This equation is always solvable. 
        \item[{\bf 2.}] Set $A_{32} := L A_{22} - A_{33} L$.  
        \item[{\bf 3.}] Set $A_{21} := c_U - A_{22} e$ and $A_{31} := c_L - A_{32} e - A_{33} e$.
        \item[{\bf 4.}] Choose $\beta \in \mathbb{R}^{p}$ to satisfy the $p$-th bushy tree order conditions, i.e., 
            \begin{align*}
                \begin{\bracChoice}
                    e & c & \cdots & c^{p-1}
                \end{\bracChoice}^T 
                \begin{\bracChoice}
                    1 & 0 \\
                    0 & -L^T \\
                    0 & I
                \end{\bracChoice} \beta = 
                \begin{\bracChoice}
                1 & \tfrac{1}{2} & \cdots & \tfrac{1}{p}     
                \end{\bracChoice}^T \, .
            \end{align*}
        \item[{\bf 5.}] Set $b$ in terms of $\beta$ via \eqref{Eq:bvector}. 
        \item[]\hspace{-0.5cm}{\bf \textrm{Remark:}} The scheme $(A,b)$ is parameterized by the inputs, has WSO $q$ and satisfies \eqref{eq:quadrature-conditions}. 
    \end{enumerate}
\end{algodef}

Then, with the remaining degrees of freedom, we seek to solve the rest of the classical order conditions.
There is an additional advantage to this approach.  It turns out that many of the classical order conditions can be expressed as linear combinations of the WSO conditions and the quadrature conditions.  Specifically, any method with WSO $q$ ($\geq p - 1$) that satisfies the quadrature conditions of order $p$ also satisfies all of the following order conditions:
\begin{equation}\label{eq:palm-condition}
	\power{b}{T}A^{k}\, c^j = \frac{1}{(k+j+1)\cdots(j+1)}
	\quad\text{for}\quad
	0\leq j+k \leq p-1\;\;\mbox{and}\;\; j\/,\,k\geq 0\, .
\end{equation}
We refer to these as \emph{palm tree conditions}, because they correspond to trees that can be formed by attaching the root of a bushy tree to the leaf of a tall tree.

For $p\le 3$, the output of Algorithm~\ref{alg:cap} already satisfies all of the required order conditions because all classical order conditions are palm tree conditions \eqref{eq:palm-condition}. For $p>3$, we must choose the remaining free parameters so as to satisfy the remaining (non-palm tree) order conditions.

\subsubsection{Choice of optimal schemes}
Even after satisfying all conditions for order $p$ and WSO $q$, some freedom typically remains in choosing the \rk{} coefficients. As the stability function for all candidate schemes is fully determined (Theorem~\ref{thm:qm3}), and since these schemes cannot be SSP, one cannot use the stability region or SSP properties to choose among these schemes.
Instead, we identify optimal schemes by minimizing a combination of the following metrics:
\begin{equation} \label{eq:principal_error}
    \mA^{(p+1)} = \sqrt{\sum_{t \in T_{p+1}} \left( \frac{1}{\sigma(t)} \left( \frac{1}{\gamma(t)} - \Phi(t) \right) \right)^2} \, ,
    \quad \textrm{and} \quad 
    \mD = \max\{|a_{i,j}|, |b_i|, |c_i|\} \, . 
\end{equation}
Here $\mA^{(p+1)}$ is the principal error norm, measuring the 2-norm of the classical order $p+1$ residuals (see \cite{Butcher2016} for details on the symbols in $\mA^{(p+1)}$), while $\mD$ quantifies the $\infty$-norm of the coefficients.  We also impose the constraint $0 \leq c_i \leq 1$ so that the ODE right hand side is always evaluated within the span of the given time step. We use $\mA^{(p+1)}$ as the primary objective function, and $\mD$ as a secondary objective function when multiple schemes achieve the same optimal $\mA^{(p+1)}$ value. 

We now present methods that satisfy \eqref{pqs-bound} with equality. By Theorem~\ref{thm:qm3} the schemes are guaranteed irreducible and have stability functions given by the $p$-th partial sum of the exponential.

\subsubsection{(3,2,2) methods}
Here $A_{22} = A_{33} = 0$, so Algorithm~\ref{alg:cap} yields $L = A_{32} = 0$ and results in a 2-parameter family of schemes. All schemes in this family have stability function $R(z) = 1 + z + z^2/2$, a linear SSP coefficient of 1, and share the same principal error norm $\mA^{(3)}$ because $b^Tc^2=b^TAc=0$. The minimum value of $\mD$ is attained at $c_2 = \frac{1}{2}$, $c_3 = 1$. The 2-parameter family, and optimal scheme are: 
\begin{equation}\label{eq:ERK(3,2,2)}
\begin{butchertableau}
{c|ccc}
0 & & \\
c_2 & c_2 & \\
c_3 & c_3 & \\
\hline
& 1\!-\!\frac{1}{2c_2}\!-\!\frac{1}{2c_3} & \frac{c_3}{2c_2(c_3-c_2)} & \frac{c_2}{2c_3(c_2-c_3)}
\end{butchertableau}
\, ,
\qquad 
    (3,2,2) \quad 
\begin{butchertableau}
{c|ccc}
0 & & \\
\frac{1}{2} & \frac{1}{2} & \\
1 & 1 & \\
\hline
& -\frac{1}{2} & 2 & -\frac{1}{2}
\end{butchertableau}\;.
\end{equation} 

\subsubsection{(4,3,2) methods}
Here $A_{22} = 0$,
$A_{33} = \begin{\bracChoice}
    0 & 0 \\
    a_{43} & 0 
\end{\bracChoice}$, and $c = \begin{\bracChoice} 0 & c_2 & c_3 & c_4 \end{\bracChoice}^T$ 
so that Algorithm~\ref{alg:cap} yields a $4$-parameter family of schemes.  Two members of this family are
\begin{equation} \label{eq:ERK(4,3,2)}
\textrm{(4,3,2)} \quad  
\begin{butchertableau}{c|cccc}
 0 \\
 \frac{3}{10} & \frac{3}{10} \\
 \frac{2}{3} & \frac{2}{3} \\
 \frac{3}{4} & -\frac{21}{320} & \frac{45}{44} & -\frac{729}{3520} \\ \hline
 & \frac{7}{108} & \frac{500}{891} & -\frac{27}{44} & \frac{80}{81} \\
\end{butchertableau}\, , \quad \quad 
\textrm{ERK312} \quad 
\begin{butchertableau}{c|cccc}
0 &  &  &  & \\
\frac{1}{2} & \frac{1}{2} &  &  & \\
1 & 1 &  &  & \\
1 & -\frac{1}{2} & 2 & -\frac{1}{2} & \\
\hline
 & \frac{1}{6} & \frac{2}{3} & -\frac{1}{6} & \frac{1}{3}\\
\end{butchertableau} \, .
\end{equation}
Here ERK312 was proposed in \cite[eq.~5.2]{skvortsov2017avoid}, while (4,3,2) is our scheme with optimal principal error. Our scheme is obtained by first minimizing $\mA^{(4)}$ with respect to the $4$ parameters $(a_{43},c_2,c_3, c_4)$. This yields a $2$-parameter family of $\mA^{(4)}$-optimal schemes, among which $\mD$ is minimized with values near $(a_{43}, c_2, c_3, c_4)=( -\tfrac{729}{3520}, \tfrac{3}{10}, \tfrac{2}{3}, \tfrac{3}{4})$.  All $(4,3,2)$ methods have $R(z)=1 + z + z^2/2 + z^3/6$ and linear SSP coefficient of 1. 

\subsubsection{(5,3,3) methods}
Algorithm~\ref{alg:cap} yields a $6$-parameter family of $(5,3,3)$ schemes since $A_{22}$ and $A_{33}$ are $2\times 2$ matrices with a single non-zero entry while $c$ has $4$ degrees of freedom. The principle error $\mA^{(4)}$ is minimized when the classical fourth order condition $b^T \operatorname{diag}(c) A c = \frac{1}{8}$ holds (the other three residuals in $\mA^{(4)}$ are independent of the free parameters). Using the remaining parameters, we numerically minimize $\mD$ to determine (5,3,3):
\begin{equation} \label{eq:ERK(5,3,3)}
\begin{aligned}
&\begin{butchertableau}{c|ccccc}
0 \\
\frac{3}{11} & \frac{3}{11} \\
\frac{15}{19} & \frac{285645}{493487} & \frac{103950}{493487} \\
\frac{5}{6} & \frac{3075805}{5314896} & \frac{1353275}{5314896} \\
1 & \frac{196687}{177710} & -\frac{129383023}{426077496} & \frac{48013}{42120} & -\frac{2268}{2405} \\ \hline
& \frac{5626}{4725} & -\frac{25289}{13608} & \frac{569297}{340200} & \frac{324}{175} & -\frac{13}{7}
\end{butchertableau}
\, , & \quad \quad 
&
\begin{butchertableau}
{c|ccccc}
0 &  &  &  &  &  \\
\tfrac{1}{3} & \tfrac{1}{3} &  &  &  &  \\
\tfrac{2}{3} & \tfrac{2}{3} &  &  &  &  \\
1 & 1 &  &  &  &  \\
0 & -\tfrac{11}{12}& \tfrac{3}{2} & -\tfrac{3}{4} & \tfrac{1}{6} &  \\ \hline
& \tfrac{1}{4} & -3 & \tfrac{15}{4} & -1 & 1
\end{butchertableau}\;. \\
&\hspace{3cm}\textrm{(5,3,3)}  & &\hspace{1.5cm}\textrm{ERK313} 
\end{aligned}
\end{equation}
An additional scheme, ERK313, from this family appears in \cite[eq. 5.2]{skvortsov2017avoid} and admits an L-shaped structure (which may have computational advantages in a parallel environment). All $(5,3,3)$ methods have $R(z)=1 + z + z^2/2 + z^3/6$, and linear SSP coefficient of 1.

\begin{remark}
It is natural to ask whether schemes with a similar L-shaped structure (where $A$ has
non-zero entries only in the first column and last row) can be constructed with $p=q=4$
or higher.  It can be shown that this is not possible. \myremarkend
\end{remark}

\subsubsection{4th and 5th order methods}\label{sec:optimized}
All schemes in the section are reported in Appendix~\ref{appendix:methods} and derived using the Integreat library \cite{roberts2023integreat}. We start with the parameterization of $(A,b)$ provided by Algorithm~\ref{alg:cap} in terms of $(A_{22}, A_{33}, c)$. 

To devise $(6,4,3)$ and $(7,4,4)$ methods, we symbolically solve the single non-palm tree condition of order four, $b^T \operatorname{diag}(c) A c = \frac{1}{8}$, and using the parameterization perform unconstrained optimization first of $\mA^{(5)}$ and then of $\mD$.

To devise $(8,5,4)$ and $(9,5,5)$ methods, we reduce the number of degrees of freedom in our parameterization by fixing the abscissae to be rational approximations of the Gauss--Lobatto nodes on $[0, 1]$ and by setting $A_{22} = 0$.
The latter assumption enables stages 2 through $q+1$ to be computed in parallel.
We use numerical optimization software to find near-optimal values of the remaining parameters, and then exactly solve the remaining non-palm tree conditions for parameters that are close to those obtained by numerical optimization.

We summarize our finding for these and lower order methods in Tables \ref{tab:min_stages} and \ref{tab:method_properties}.

\begin{SCtable}[2.5]
\centering
\begin{tabular}{l|ccccc}
\diagbox{q}{p} & 2 & 3 & 4 & 5 \\ \hline
1     & \dbox{2} & 3 & 4 & 6 \\
2     & \fbox{3} & \dbox{4} & 5 & 6 \\
3     & 4 & \fbox{5} & \dbox{6} & 7 \\
4     & 5 & 6 & \fbox{7} & \dbox{8} \\
5     & 6 & 7 & 8 & \fbox{9} \\
\end{tabular}
\caption{Minimum number of stages required to achieve a given $p$, $q$ up to $p,q\le 5$. The $q = 1$ row represents the well-known bounds for \erk{} methods up to order 5 \cite{hairer1993solving}. The $2\le q\le 5$ rows were determined by finding, for each $(p,q)$-pair, a concrete \erk{} scheme that satisfies the proven (Thm.~\ref{thm:spq}) lower bound $s\ge p + q - 1$. The column $p = 1$ is not displayed, because the explicit Euler method has WSO $q = \infty$. Solid boxes mark the schemes natural for ODE problems ($q = p$), while dashed boxes denote the schemes of interest for PDE problems ($q = p-1$).
}
\label{tab:min_stages}
\end{SCtable}

\begin{table}[ht!]
    \centering
    \begin{tabular}{r|cccccc}
        Method & Source & Stages & Order & WSO & $\mA^{(p+1)}$ & $\mD$ \\ \hline \hline
        \textbf{(3,2,2)} & \eqref{eq:ERK(3,2,2)} & 3 & 2 & 2 & 2.357E-1 & 2 \\ \hline
        3\textsuperscript{rd} Order Shu--Osher & \cite{shu1988efficient} & 3 & 3 & 1 & 7.217E-2 & 1 \\
        \textbf{(4,3,2)} & \eqref{eq:ERK(4,3,2)} & 4 & 3 & 2 & 5.893E-2 & 1.003 \\
        ERK312 & \cite{skvortsov2017avoid} & 4 & 3 & 2 & 7.217E-2 & 2 \\
        \textbf{(5,3,3)} & \eqref{eq:ERK(5,3,3)} & 5 & 3 & 3 & 7.217E-2 & 1.858 \\
        ERK313 & \cite{skvortsov2017avoid} & 5 & 3 & 3 & 1.443E-1 & 3.75 \\ \hline
        Classical 4\textsuperscript{th} Order \rk{} & \cite{kutta1901beitrag} & 4 & 4 & 1 & 1.450E-2 & 1 \\
        \textbf{(6,4,3)} & \ref{sec:ERK(6,4,3)} & 6 & 4 & 3 & 1.443E-2 & 1.144 \\
        \textbf{(7,4,4)} & \ref{sec:ERK(7,4,4)} & 7 & 4 & 4 & 1.667E-2 & 6.187 \\ \hline
        5\textsuperscript{th} Order Dormand--Prince & \cite{dormand1980family} & 7 & 5 & 1 &  3.991E-4 & 11.60 \\
        \textbf{(8,5,4)} & \ref{sec:ERK(8,5,4)} & 8 & 5 & 4 & 1.217E-2 & 25.33 \\
        \textbf{(9,5,5)} & \ref{sec:ERK(9,5,5)} & 9 & 5 & 5 & 3.316E-2 & 44.42 
    \end{tabular}
    \caption{Properties of methods presented in this paper (bold) and comparable method from the literature. Sources with equation or section numbers (i.e., B.x) refer to the present work.}
    \label{tab:method_properties}
\end{table}

\subsection{Parallel iterated \rk{} methods with arbitrarily high WSO} \label{sec:high}
Here we systematically construct parallel iterated \rk{} \cite[Section II.11]{hairer1993solving} methods with arbitrarily high WSO.  While in principle, one could use the approach in \cref{subsec:conERKslow} to construct schemes with $p,q \geq 6$, the approach here provides a simple alternative to simultaneously solve both the WSO and order conditions for arbitrary $p, q$ (including the non-palm tree order conditions). Although the schemes in this section do not come close to satisfying the optimum order $p + q = s+1$, they demonstrate that \erk{} methods with arbitrarily high WSO exist. Furthermore, in an ideal parallel computational environment, the resulting parallel iterated \rk{} methods (e.g., in \cref{thm:arbitrary_order}) 
are roughly as expensive as a serial, $p$-stage, explicit \rk{} method, and could have practical advantages.

Like extrapolation and deferred correction, parallel iteration extends a basic \rk{} method (typically low order) to a new method of higher classical order.  Generally, parallel iterated techniques do not increase the WSO beyond that of the basic method; however with a specially-crafted basic method, we show that high WSO is attainable.

\begin{theorem} \label{thm:arbitrary_order}
    For all $p \geq 2$, there exist \erk{} methods of the type $(p^2, p, p)$. 
\end{theorem}

\begin{proof}
    The proof is constructive and provides a family of methods parameterized by a set of abscissae.  For any $p \geq 2$, we take $\widetilde{c} \in \mathbb{R}^{p+1}$ to be an arbitrary set of distinct abscissae. We then take the parallel iterated \rk{} method $(A,b)$ to be given in terms of a $(p+1)$-stage basic method $(\widetilde{A}, \widetilde{b})$ as follows (cf. \cite[eq.~11.6, Section~II.11]{hairer1993solving}):
\begin{equation*}
    \begin{butchertableau}{r|l}
        {c} & A \\ \hline
        & {b}^T
    \end{butchertableau}
    =
    \begin{butchertableau}{c|ccccc}
        0 & 0 \\
        \widetilde{c} & \widetilde{A} & 0 \\
        \widetilde{c} & 0 & \widetilde{A} & 0 \\
        \vdots & \vdots & & \ddots & \ddots \\
        \widetilde{c} & 0 & \cdots & 0 & \widetilde{A} & 0 \\ \hline
        & 0 & \cdots & 0 & 0 & \widetilde{b}^T
    \end{butchertableau}
    \qquad \text{with} \qquad
    \begin{aligned}
        \widetilde{A} &= \Vtilde \Stilde \Vtilde^{-1} \\
        \widetilde{b}^T &= e^T \Stilde \Vtilde^{-1} \\
        \widetilde{V} &= \lbrac\begin{array}{c|c|c|c} 
            e &  \widetilde{c} & \cdots &  \widetilde{c}^p
        \end{array}\rbrac \\
        S &= \begin{bmatrix}
            0 \\
            1 & 0 \\
            & \frac{1}{2} & 0 \\
            & & \ddots & \ddots \\
            & & & \frac{1}{p} & 0
    \end{bmatrix}
    \end{aligned}
\end{equation*}

    Here the form of $(\widetilde{A}, \widetilde{b})$ is a special case of the \emph{V-transformation} in \cite{burrage1978special} (cf.~\cite[Section~IV.5]{hairer1996solving}). We take the number of row and column blocks in $A$ to be $p$, which corresponds to the number of parallel iterations. Thus, the method has $p^2+p$ stages, but it is reducible (by consolidating the first $p+1$ stages) to a $p^2$-stage method. 
    
    By construction, the basic method $(\widetilde{A}, \widetilde{b})$ has stage order and classical order $p$. By Theorem 11.2 in \cite[Section~II.11]{hairer1993solving}, the parallel iterated method $(A,b)$ also has classical order $p$. It remains to show that $(A,b)$ has WSO $p$ by checking \eqref{eq:WSO_rational} for $k = 2, \dots, p$: 
    \begin{align*}
        -k \, b^T (I - z A)^{-1} \tau^{(k) }
        &= \begin{\bracChoice}
            0 & \cdots & 0 & \widetilde{b}^T
        \end{\bracChoice} \begin{\bracChoice}
            I \\
            z \widetilde{A} & I \\
            \vdots & \ddots & \ddots \\
            z^{p-1} \widetilde{A}^{p-1} & \cdots & z \widetilde{A} & I
        \end{\bracChoice} \begin{\bracChoice}
            0 \\ \, \widetilde{c}^{\,k} \\ 0 \\ \vdots \\ 0
        \end{\bracChoice} = z^{p-2} \, \widetilde{b}^T \widetilde{A}^{p-2} \, \widetilde{c}^{\, k} \\
        &=  z^{p-2} \, e^T \Stilde^{p-1} \Vtilde^{-1} \widetilde{c}^{\, k}
        = z^{p-2} \, e^T \Stilde^{p-1} e_{k + 1}
        = 0 \, . 
    \end{align*}    
\end{proof}

\section{Efficient implementation for linear constant-coefficient problems with time-dependent forcing} \label{sec:implementation}
High WSO schemes have more stages per time step, and thus may be more costly, than low WSO schemes. We will show that high WSO schemes applied to linear ODEs 
\begin{equation} \label{eq:linear_ODE}
    y' = L y + g(t)\, ,
\end{equation}
may be implemented with only $\dim Y$ evaluations of (the potentially costly) $Ly$ as opposed to the usual $s$ ($g$ still requires $s$ evaluations per step).  This implementation is a specific case of a \emph{generalize additive Runge--Kutta} (GARK) scheme \cite{sandu2015generalized,roberts2022eliminating}, and can be viewed as a $\dim(Y)$-stage method with modified forcing. Here \eqref{eq:linear_ODE} includes semi-discretizations of constant-coefficient linear PDEs with time-dependent boundary conditions and forcings.
\begin{theorem} \label{thm:GARK_form}
    When applied to \eqref{eq:linear_ODE}, an \erk{} scheme $(A, b)$ with order $p \geq 1$ and $d = \dim(Y)$ is equivalent to the following two-additive GARK method \cite[eq. 8]{roberts2022eliminating}:    
    \begin{subequations} \label{eq:GARK}
        \begin{alignat}{2}
            \label{eq:GARK:stages}
            Y_i &= y_n + \dt \, \sum_{j=1}^{d} \, \widehat{a}_{i,j}\, L Y_j + \dt \, \sum_{j=1}^{s} 
            \, \rktwo{a}_{i,j} \, g(t_n + \rktwo{c}_j \dt) \, , & \qquad 1 \leq i \leq d \, , \\
            \label{eq:GARK:solution}
            y_{n+1} &= y_n + \dt \, \sum_{j=1}^{d} \, \rkone{b}_{j} \, L Y_j + \dt \, \sum_{j=1}^{s} \, \rktwo{b}_{j} \, g(t_n + \rktwo{c}_j \dt) \, .
        \end{alignat}
    \end{subequations}
    Here $\dt$ is the time step, the stages are consistent (i.e., $\rktwo{c} = A e = c$) and the coefficients are 
    \begin{equation*}        
            \rkone{A} = \begin{\bracChoice}
                0 \\
                b^T A^{d-1} e & 0 \\
                b^T A^{d-2} e - 1 & 1 & 0 \\
                b^T A^{d-3} e - 1 & 0 & 1 & 0 \\
                \vdots & \vdots & & \ddots & \ddots \\
                b^T A^{1} e - 1 & 0 & 0 & \cdots & 1 & 0
            \end{\bracChoice}
             \quad  \rkone{b} = e_d, \quad\textrm{and} \quad  
            \rktwo{A} = \begin{\bracChoice}
                0 \\
                b^T A^{d-1} \\
                b^T A^{d-2} \\
                \vdots \\
                b^T A
            \end{\bracChoice} \quad 
            \rktwo{b} = b \, .
    \end{equation*}
\end{theorem}
\begin{proof}
    We can solve the stages \eqref{eq:GARK} for $Y$ in compact vector form and write
    \begin{align*}
        Y &= \left(I - \dt \rkone{A} \, \otimes L \right)^{-1} \, \left(e \otimes y_n + \dt \, ( \rktwo{A} \otimes I) \, g(t_n + \rktwo{c} \dt) \right), \\
        y_{n+1} &= y_n + \dt \,\left( ( \rkone{b}^T \otimes L ) \, Y
         + ( \rktwo{b}^T \otimes I) \, g(t_n + \rktwo{c} \dt)  \right) \, ,  
    \end{align*}
    where $\otimes$ denotes a Kronecker product. Neumann-expanding the matrix inverse in $Y$ yields: 
    \begin{align*}
        y_{n+1} &= y_n + \sum_{i=1}^{d} \,
        \underbrace{\rkone{b}^T \rkone{A}^{i-1} e}_{=b^T A^{i-1}e} \, (\dt L)^{i} \, y_n 
        + \dt \, \left( (b^T \otimes I) + \sum_{i=1}^{d} \, \underbrace{\rkone{b}^T \rkone{A}^{i-1} \rktwo{A}}_{=b^T A^i} \, \otimes \, (\dt L)^i \right) \, g(t_n + c \dt)         \, \\
        &= R(\dt L) \, y_n + \dt \, (b^T \otimes I) \, (I - A \otimes \dt L)^{-1} g(t_n + c \dt)\, ,
    \end{align*} 
    which is exactly the numerical solution produced by the method $(A, b)$ applied to \eqref{eq:linear_ODE}. Here we have used identities for $\rkone{b}^T \rkone{A}^{i-1} e$ and $\rkone{b}^T \rkone{A}^{i-1} \rktwo{A}$ which follow from: For $1 \leq i \leq d - 1$
    \begin{align*}
        \rkone{b}^T \, \rkone{A}^{i-1} = \begin{\bracChoice}
            b^T A^{i-1} e - 1 & 0_{1 \times (d-i-1)} & 1 & 0_{1 \times (i-1)}
        \end{\bracChoice} \quad \textrm{and} \quad
         \rkone{b}^T \, \rkone{A}^{d-1} = \begin{\bracChoice}
            b^T A^{d-1} e & 0 & \cdots & 0 
        \end{\bracChoice} .         
    \end{align*}     
\end{proof}

\section{Numerical tests} \label{sec:tests}
We now study the convergence rate of the methods listed in Table~\ref{tab:method_properties} on a range of test problems. Our new methods are denoted simply by the triplet $(s, p, q)$, as discussed in \cref{sec:construction}: $(4, 3, 2)$, $(5, 3, 3)$, $(6, 4, 3)$, $(7, 4, 4)$, $(8, 5, 4)$, and $(9, 5, 5)$. To assess their performance, we compare these methods against well-known \erk{} methods; like virtually all existing \erk{} methods, these methods have WSO $q = 1$. Specifically, we consider the third order Shu--Osher method \cite{shu1988efficient}, the classical fourth-order method of Kutta \cite{kutta1901beitrag}, and the fifth-order Dormand--Prince method \cite{dormand1980family}, referred to as $(3, 3, 1)$, $(4, 4, 1)$, and $(7, 5, 1)$ respectively, for the purpose of these comparative assessments. 

The primary purpose of this section is to validate the theoretical results regarding order of convergence in the presence of time-dependent boundary conditions. It is not intended to be a study of the relative efficiency of the various methods involved, since efficiency is highly implementation-dependent and ought to involve careful parallel implementations of some of the methods. This is the subject of future work. Implementations for all the numerical examples presented here can be found in \cite{ERK_High_WSO_Code}.

\subsection{Spatial discretization and boundary conditions}
Since the phenomenon of order reduction (as well as various proposals for its cure)
is intimately related to the manner in which boundary conditions are imposed, it
is important to state precisely how we impose boundary conditions in the following
tests.

In the first two tests, which use first-order upwind finite difference spatial 
discretizations in one dimension, we simply impose the value of the boundary
condition function at the leftmost grid point, which coincides with the boundary.

In the remaining tests, which make use of the finite volume software Clawpack \cite{ketcheson2012pyclaw,mandli2016clawpack}, we utilize Clawpack's SharpClaw solver algorithm and replace the WENO reconstruction in space with a simple centered fifth-order finite volume reconstruction. 
Limiting is not needed for these test problems since the manufactured solutions are smooth (no shocks form).
For the 2D linear acoustics system, using a linear reconstruction instead of WENO ensures that the semi-discretization is linear.
We impose the boundary conditions in the traditional way, using ghost cells situated just outside the problem domain.  At each \rk{} stage and in each ghost cell, we set the solution value equal to the cell average of the exact solution at the corresponding time.

\subsection{Linear advection equation}
\label{subsec:linear_advection}
We begin with a simple hyperbolic model equation 
\begin{equation} \label{advection}
     \begin{aligned}
        u_t & = -u_x + f(x,t), \quad 0\leq x \leq 1,\quad 0 \leq t \leq 0.7 \;, \\
        u(x,0) & = u_{0}(x) \;, \\
        u(0,t) & = g_{0}(t) \;,
    \end{aligned}  
\end{equation}
with $u_0(x) = 1+x$, $f(x,t) = \frac{t-x}{(1+t)^2}$, and $g_0(t) = \frac{1}{1+t}$ such that the exact solution becomes $u(x,t) = \frac{1+x}{1+t}$. This test problem was used to demonstrate and analyze the order reduction phenomenon in \cite{sanz1986convergence}. We apply a first-order upwind finite difference space discretization on a uniform grid. Because the true solution is linear in $x$, no spatial discretization errors arise. Consequently, the error in the fully-discrete numerical solution results solely from the time discretization. 

\begin{figure}
    \centering
	\begin{minipage}[b]{.3\textwidth}
		\includegraphics[width=\textwidth]{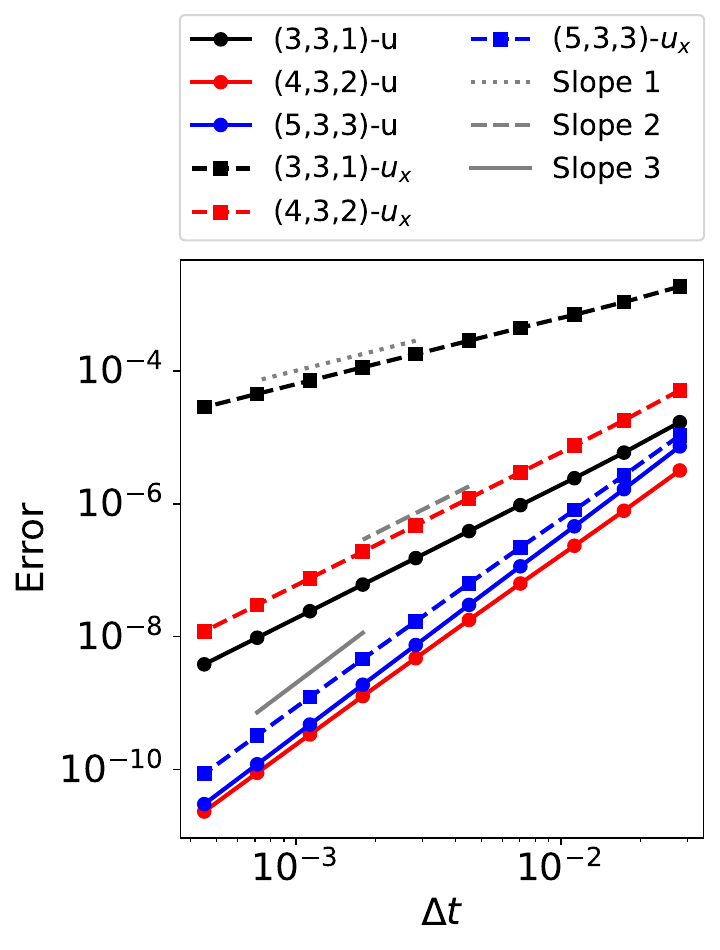}
	\end{minipage}
	\begin{minipage}[b]{.3\textwidth}
	\includegraphics[width=\textwidth]{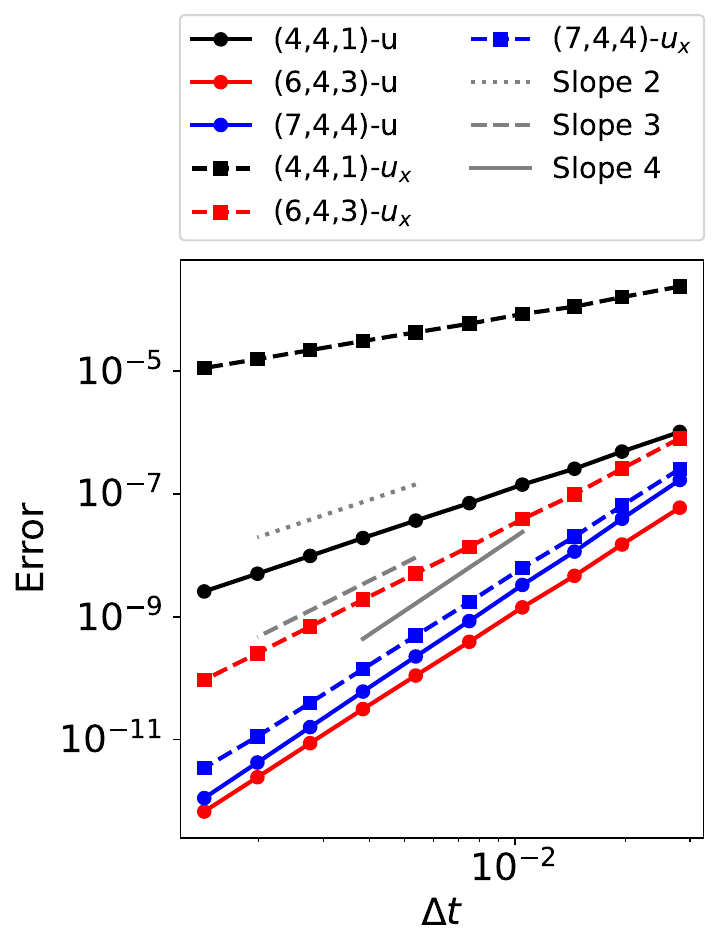}
	\end{minipage}
	\begin{minipage}[b]{.3\textwidth}
    \includegraphics[width=\textwidth]{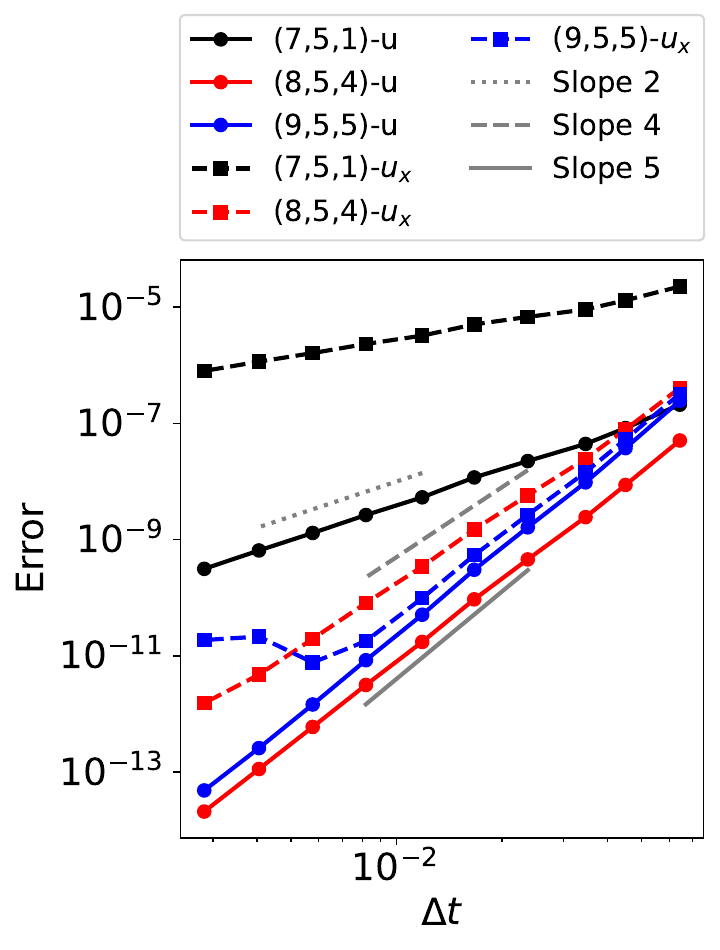}
	\end{minipage}
    \vspace{-.7em}
	\caption{Error convergence for the linear advection problem \eqref{advection}. Solid lines show convergence of the solution and dashed lines show convergence of the spatial derivative. Left: 3rd-order methods from this work (red and blue) compared with the 3rd-order Shu-Osher method. Middle: 4th-order methods from this work vs.\ the classical method of Kutta. Right: 5th-order methods from this work vs.\ the Dormand--Prince method.}
	\label{fig:LinAdv}
\end{figure} 

For each time stepping scheme, we solve the semi-discrete system up to $t=0.7$ using a constant CFL number of $0.9$. Figure~\ref{fig:LinAdv} presents the convergence curves for both the solution $u$ and its spatial first derivative $u_x$ (which is calculated via a 6th order finite difference formula based on grid-values of $u$). The reason for considering $u_x$ is as follows. A key mechanism for order reduction for IBVPs is the formation of numerical boundary layers \cite{rosales2017spatial, LubichOstermann1995} in the error. WSO reduces this boundary layer's amplitude, thereby recovering high order accuracy, but it does not remove it.

Figure~\ref{fig:LinAdv} shows that the traditional methods exhibit order reduction, reducing to second-order convergence in $u$ and first-order convergence in $u_x$. In contrast, all the high WSO methods exhibit the full order of convergence in $u$. Moreover, the methods with WSO $q = p-1$ lose one order in $u_x$. The reason for the loss of one order for schemes with $q < p$ is that for this first-order problem, the time-stepping scheme produces boundary layers of thickness $\mathcal{O}(\Delta t)$.

\subsection{Inviscid Burgers' equation} 
As a first example of a nonlinear PDE problem, we consider the inviscid Burgers' equation
\begin{equation} \label{burgers}
     \begin{aligned}
        u_t + u u_x & =  0, \quad 0\leq x \leq 1,\quad 0 \leq t \leq 0.8 \;, \\
        u(x,0) & = u_{0}(x) \;, \\
        u(0,t) & = g_{0}(t) \;, \\
    \end{aligned}  
\end{equation}
with $u_0(x) = 1+x$, and $g_0(t) = \frac{1}{1+t}$ such that the exact solution becomes $u(x,t) = \frac{1+x}{1+t}$. 
We again use a first-order upwind finite difference approximation of $u_x$ in space, which once again does not contribute to the overall error because the exact solution is linear in $x$.
Since this is a nonlinear problem, it falls outside the class of problems for which WSO is known to prevent order reduction.

\begin{figure}
	\begin{minipage}[b]{.3\textwidth}
		\includegraphics[width=\textwidth]{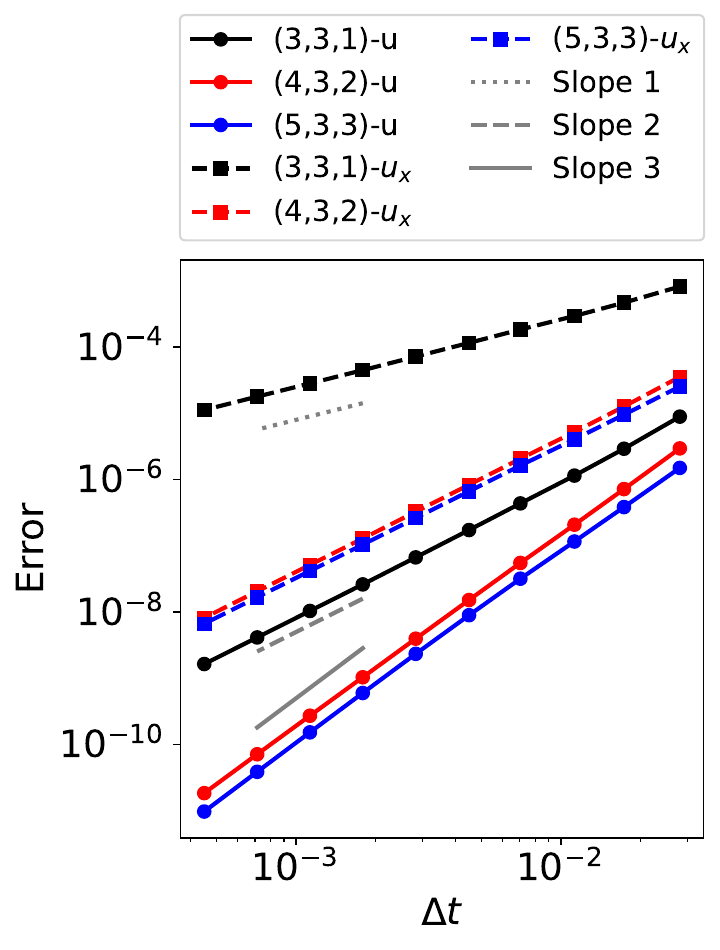}
	\end{minipage}
	\begin{minipage}[b]{.3\textwidth}
	\includegraphics[width=\textwidth]{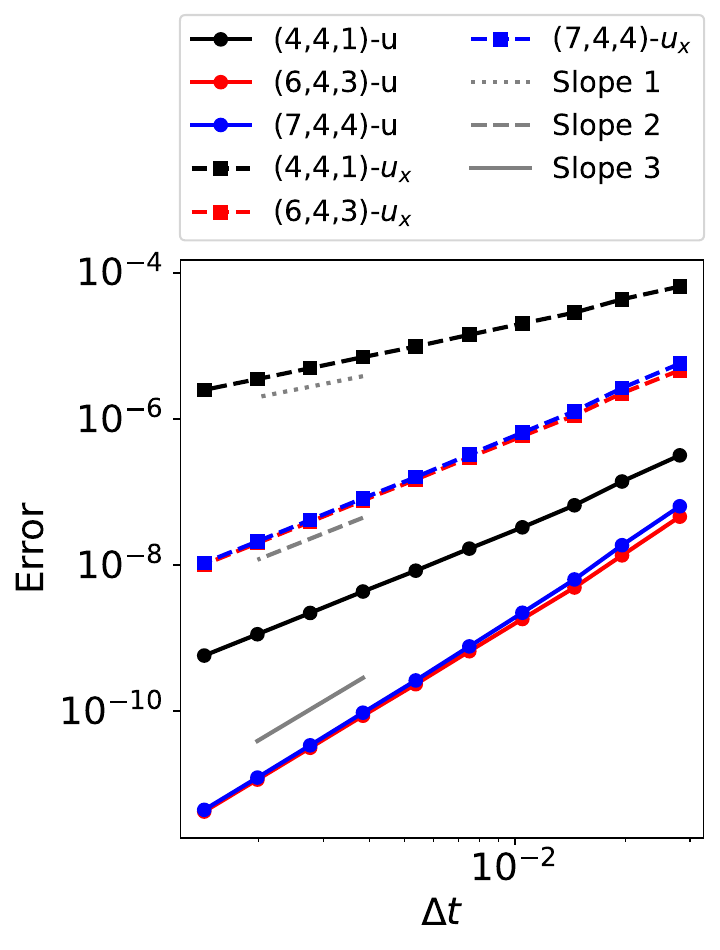}
	\end{minipage}
	\begin{minipage}[b]{.3\textwidth}
    \includegraphics[width=\textwidth]{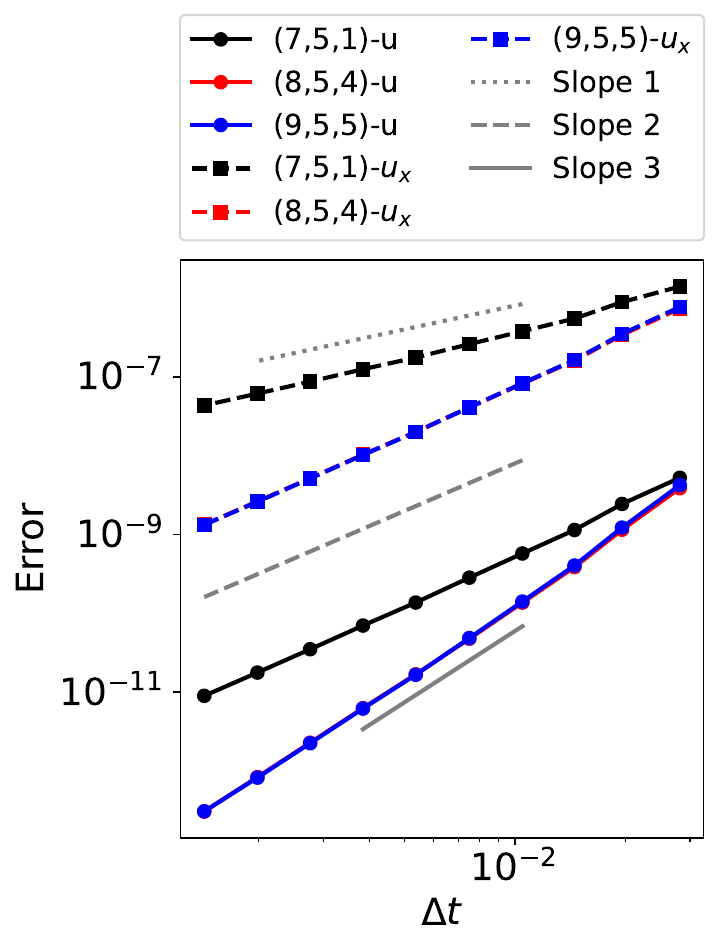}
	\end{minipage}
    \vspace{-.7em}
	\caption{Error convergence for the inviscid Burgers' problem \eqref{burgers} at $t=0.8$. Left: Third-order methods. Middle: Fourth-order methods. Right: Fifth-order methods.}
	\label{fig:ErrConvg_Burgers}
\end{figure} 
We integrate up to $t=0.8$ with a CFL number of $0.9$.
Convergence results are presented in Figure~\ref{fig:ErrConvg_Burgers}. All the classical methods with WSO 1 experience order reduction, resulting in a convergence rate of two for $u$ and one for $u_x$.  All the high WSO methods exhibit 
a convergence order of 3 or higher for $u$, and an order of 2 for $u_x$. 
We see that methods with high WSO still exhibit an advantage over traditional methods.
The reduction to order 3 by all the high WSO schemes considered here aligns with previous observations in \cite{biswas2023design, ketcheson2020dirk}.

\subsection{1D shallow water equations with a manufactured solution}
We next consider a 1D nonlinear system of hyperbolic equations: the shallow-water equations:
\begin{subequations}
    \begin{align} 
    h_{t}+{(hu)}_x &= f^h \;, \\ 
    (hu)_{t}+{\left(hu^2+\frac{1}{2}gh^2\right)}_x & = f^{hu}\;,
\end{align}
\end{subequations}
where $h$ is the total fluid column height as a function of $x$ and $t$, $u$ is the fluid's horizontal flow velocity, and $g$ is acceleration due to gravity. 
We consider the domain $x\in[0,1], t \in [0,0.5]$, and choose
the forcing function, the initial conditions, and the Dirichlet boundary conditions 
such that the solution is
\begin{align*} 
    h(x,t) = \frac{1+x}{1+t}\;, \quad  u(x,t)  = \frac{1+x^2}{0.5+t} \;.
\end{align*}
\begin{figure}[htb]
	\begin{minipage}[b]{.31\textwidth}
		\includegraphics[width=\textwidth]{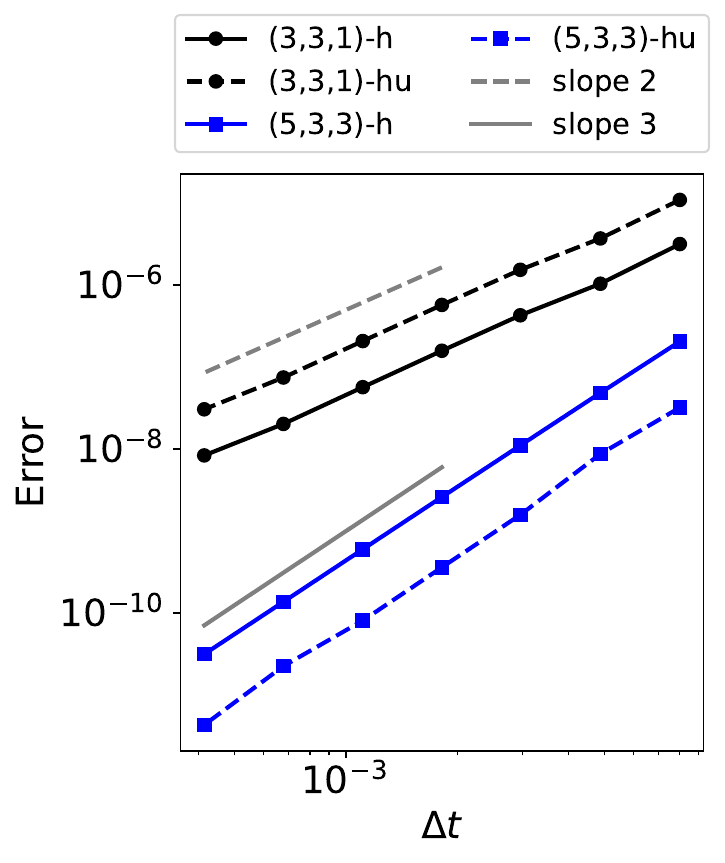}
	\end{minipage}
	\begin{minipage}[b]{.31\textwidth}
	\includegraphics[width=\textwidth]{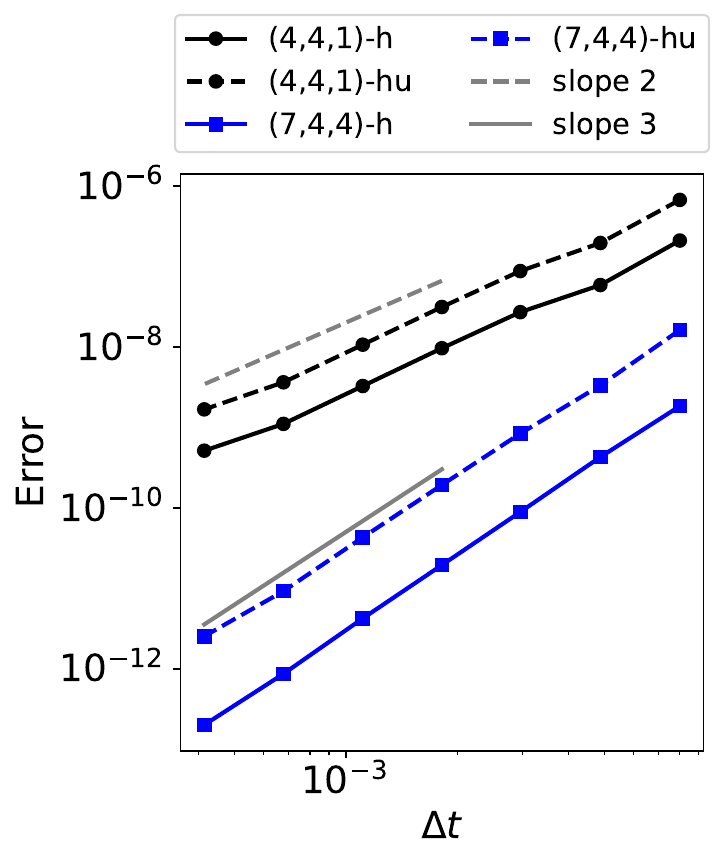}
	\end{minipage}
	\begin{minipage}[b]{.31\textwidth}
    \includegraphics[width=\textwidth]{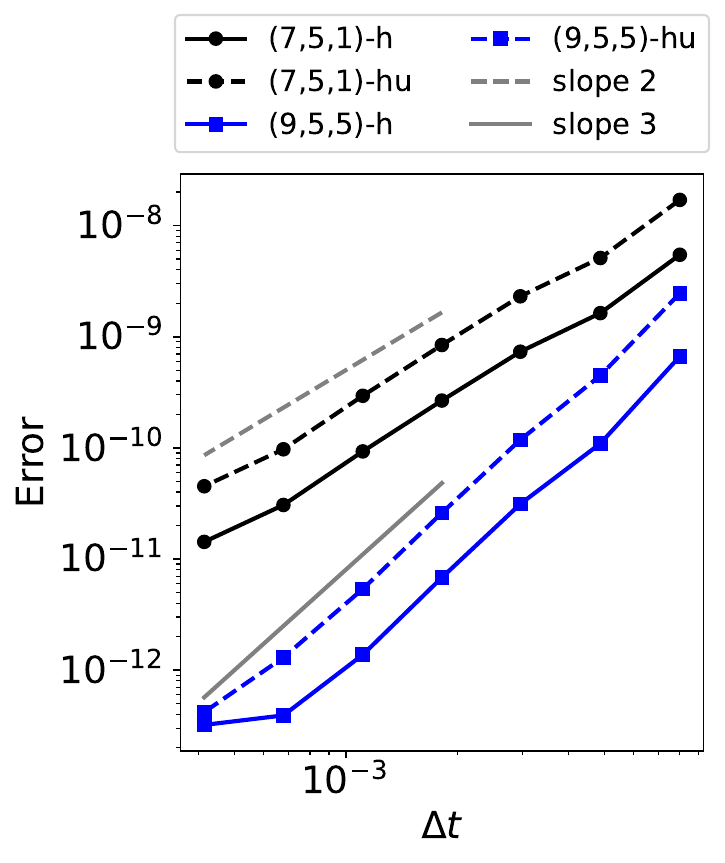}
	\end{minipage}
    \vspace{-.7em}
	\caption{Error convergence for the 1D shallow-water equations at $t=0.5$. Left: Third-order methods with different WSO. Middle: Fourth-order methods with different WSO. Right: Fifth-order methods with different WSO.}
	\label{fig:ErrConvg_1D_SWE}
\end{figure} 
To solve this problem, we employ the wave-propagation software Clawpack with a 5th-order spatial discretization, as described above.
We solve the equations with a constant CFL number of 0.8. Figure~\ref{fig:ErrConvg_1D_SWE} shows the convergence at $t=0.5$. The results are similar to those for Burgers' equation: the order of convergence for traditional methods is reduced to two, while all the high WSO methods exhibit a convergence order of three.  The high WSO methods exhibit substantially smaller errors across the full range of grid sizes considered.

\subsection{2D linear acoustics equations}
Because order reduction for PDE can manifest via numerical boundary layers (see \cref{subsec:linear_advection} and \cite{rosales2017spatial, LubichOstermann1995}), we here also investigate problems beyond 1D, because more interesting boundary layer structures are possible in higher dimensions. It is important to stress that the new \erk{} schemes mitigate order reduction solely via the coefficients of the time discretization and thus their application is oblivious to the spatial properties. This is in contrast to alternative approaches based on modified or special treatments of boundary condition, which have largely been tested in 1D and whose adaptation to higher dimensions may not be straightforward. Here we consider the 2D acoustic wave equation
\begin{subequations}
\begin{align} 
    p_{t}+K_{0}u_x+K_{0}v_y & = f^{p} \;, \\ 
    u_{t}+\frac{1}{\rho_{0}}p_x & = f^{u} \;,\\
    v_{t}+\frac{1}{\rho_{0}}p_y & = f^{v} \;,
\end{align} 
\end{subequations}
where $p(x,y,t)$ represents pressure and $(u(x,y,t),v(x,y,y))$ represents velocity. The parameter $\rho_{0}$ denotes the density of the medium, and $K_{0}$ represents the bulk modulus of compressibility. We choose $\rho_{0} = 1$ and $K_{0} = 4$. Once again, we manufacture a solution, where we select the forcing functions and the time-dependent Dirichlet boundary conditions such that the solution is
\begin{align*} 
    p(x,y,t) & = \frac{1+x+y}{1+t}\;, \quad  u(x,y,t)  = \frac{xt}{1+t} \;, \quad v(x,y,t)  = \frac{yt}{1+t} \;. 
\end{align*}
The spatial domain is taken as the unit square $[0,1]\times[0,1]$.

\begin{figure}
	\begin{minipage}[b]{.32\textwidth}
		\includegraphics[width=\textwidth]{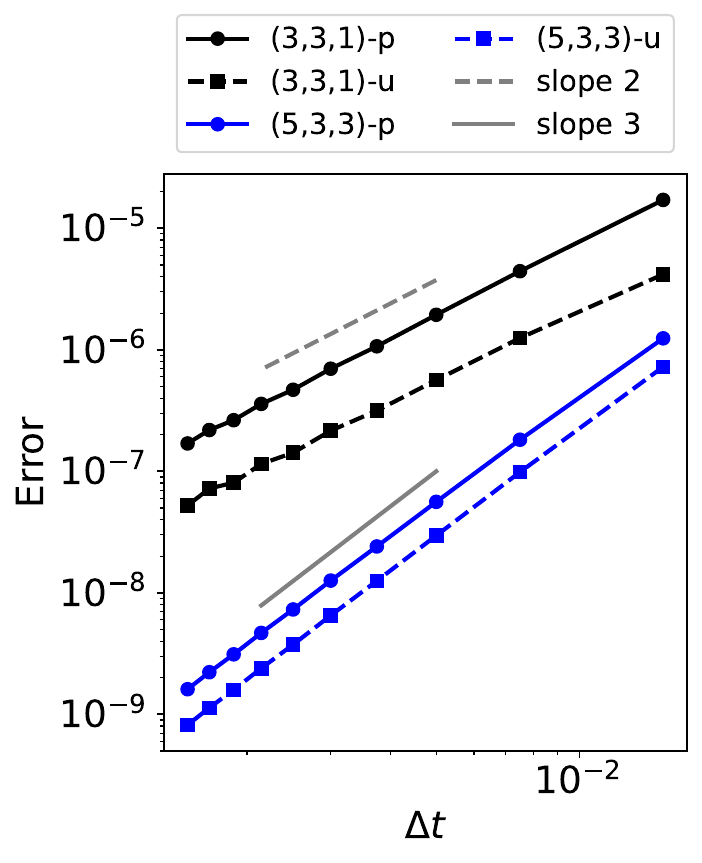}
	\end{minipage}
	\hfill
	\begin{minipage}[b]{.32\textwidth}
	\includegraphics[width=\textwidth]{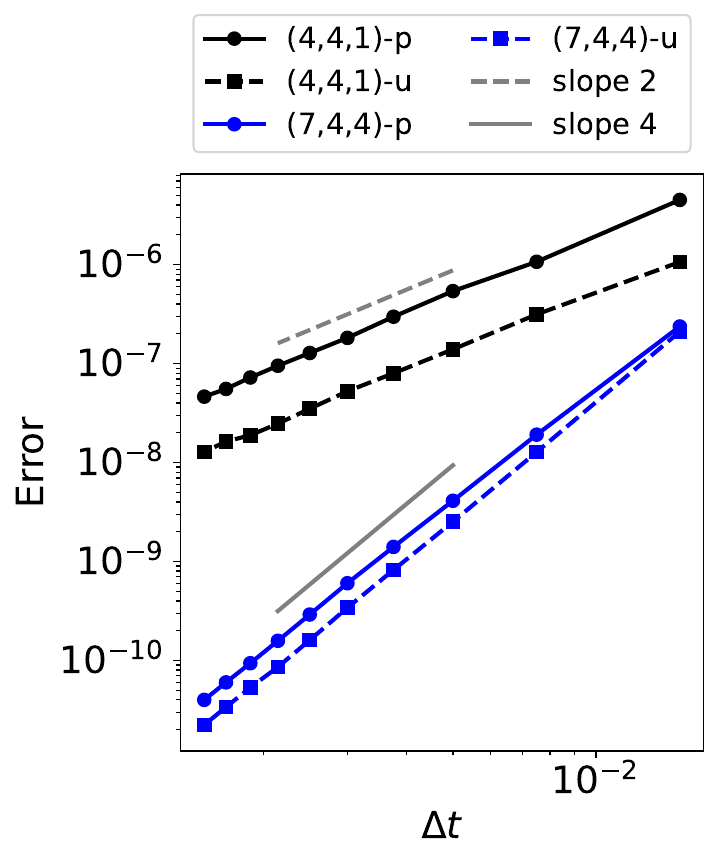}
	\end{minipage}
	\hfill
	\begin{minipage}[b]{.32\textwidth}
    \includegraphics[width=\textwidth]{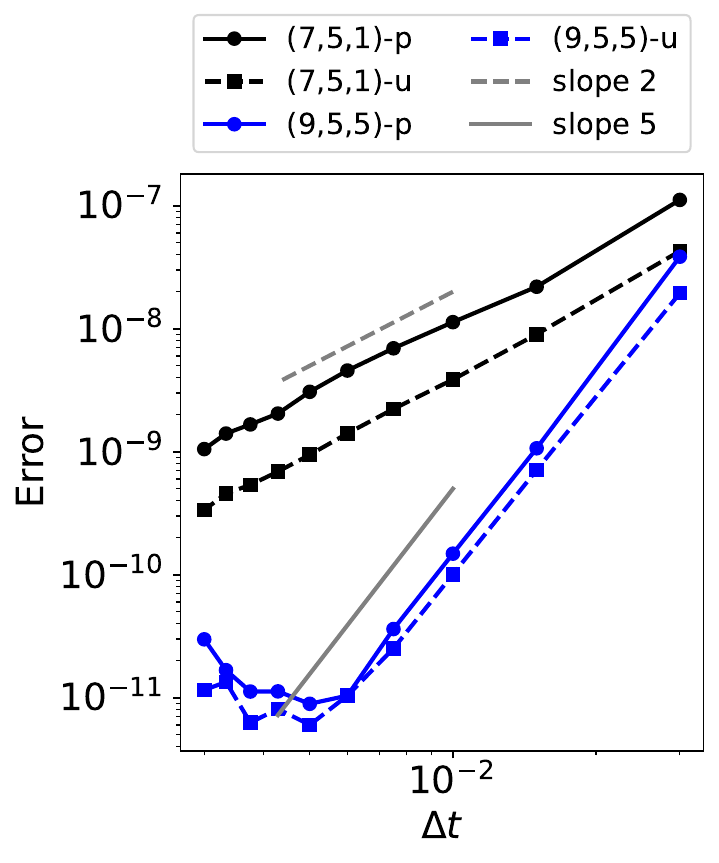}
	\end{minipage}
    \vspace{-.7em}
	\caption{Error convergence for the 2D linear acoustics equations at $t=1$. Left: Third-order methods with different WSO. Middle: Fourth-order methods with different WSO. Right: Fifth-order methods with different WSO. For all methods, the errors in the y-velocity ($v$) closely match those of the x-velocity ($u$), so they are omitted here.}
	\label{fig:ErrConvg_2D_Lin_Acoustics}
\end{figure} 

We again use Clawpack with a fifth-order spatial discretization to solve this problem.  We integrate with a CFL number of $0.45$ until $t=1$. The error behavior is shown in Figure~\ref{fig:ErrConvg_2D_Lin_Acoustics}.   As usual, all traditional methods suffer from order reduction. In contrast, the high WSO methods, in accordance with the theory, achieve their respective full order of convergence and produce significantly smaller errors compared to their WSO 1 counterparts.

\subsection{2D shallow water equations: a parabolic flood wave}
Our final test problem is a classical benchmark proposed by Thacker \cite{thacker1981some}, in which a parabolic mound of water spreads over a frictionless surface (we do not include Coriolis forces). We solve the 2D shallow water equations
\begin{subequations}
\begin{align} 
    h_{t}+{(hu)}_x+{(hv)}_y & = 0 \;, \\ 
    (hu)_{t}+{\left(hu^2+\frac{1}{2}gh^2\right)}_x+{(huv)}_y & = 0  \;,\\
    (hv)_{t}+{(huv)}_x+{\left(hv^2+\frac{1}{2}gh^2\right)}_y & = 0 \;,
\end{align} 
\end{subequations}
where $h(x,y,t)$ represents the fluid column height, $(u(x,y,t),v(x,y,t))$ corresponds to the fluid's horizontal flow velocity, and $g$ denotes the acceleration due to gravity. The solution (cf.~\cite{thacker1981some}) is
\begin{subequations}
\begin{align*} 
    h(x,y,t)  = \eta \left[\frac{T^2}{t^2+T^2}-\frac{x^2+y^2}{R_0^2}\left(\frac{T^2}{t^2+T^2}\right)^2\right] \, , \quad 
    u(x,y,t)  = \frac{xt}{t^2+T^2} \, , \quad 
    v(x,y,t)  = \frac{yt}{t^2+T^2} \, ,
\end{align*}   
\end{subequations}
where $\eta$ represents the initial height of the central point of the parabolic mound, $R_{0}$ is the initial radius of the mound, and $T$ is the time taken for the central height to decrease to $\frac{\eta}{2}$. We take $g=1$, $T=1$, and $\eta = 2$, such that
$R_0 = T \sqrt{2g\eta}=2$ \cite{thacker1981some}. 
The spatial domain is $\Omega = [-0.5, 0.5] \times [-0.5, 0.5]$, so that $h$ is strictly positive everywhere and the solution is smooth.
 
\begin{figure}
	\begin{minipage}[b]{.31\textwidth}
		\includegraphics[width=\textwidth]{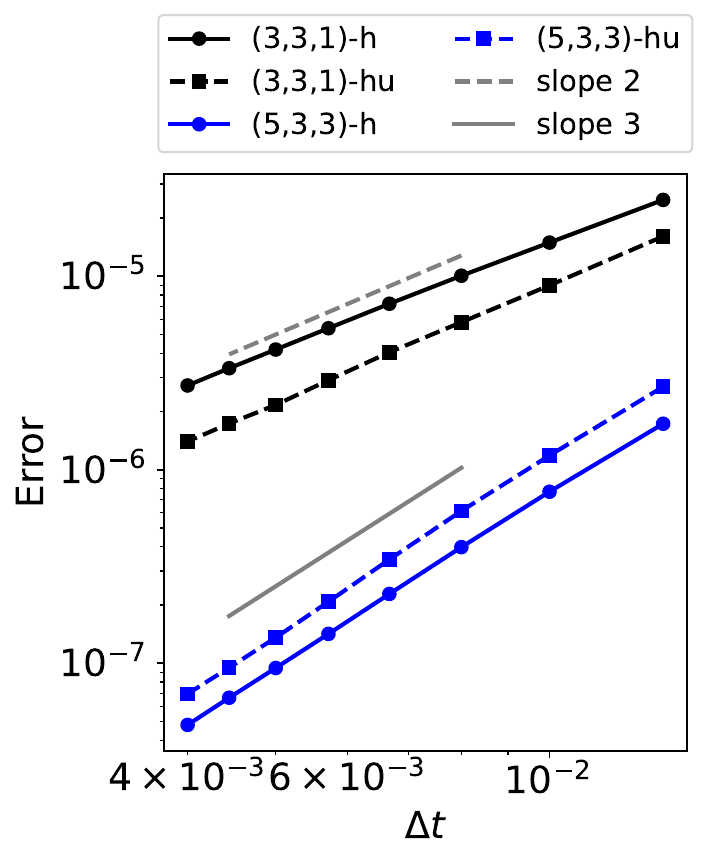}
	\end{minipage}
	\begin{minipage}[b]{.31\textwidth}
	\includegraphics[width=\textwidth]{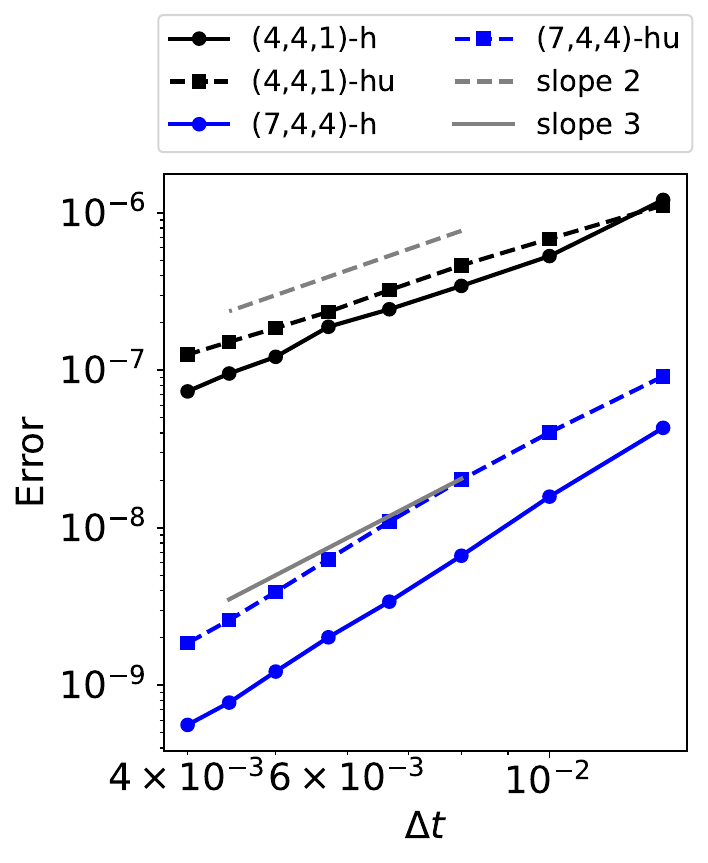}
	\end{minipage}
	\begin{minipage}[b]{.31\textwidth}
    \includegraphics[width=\textwidth]{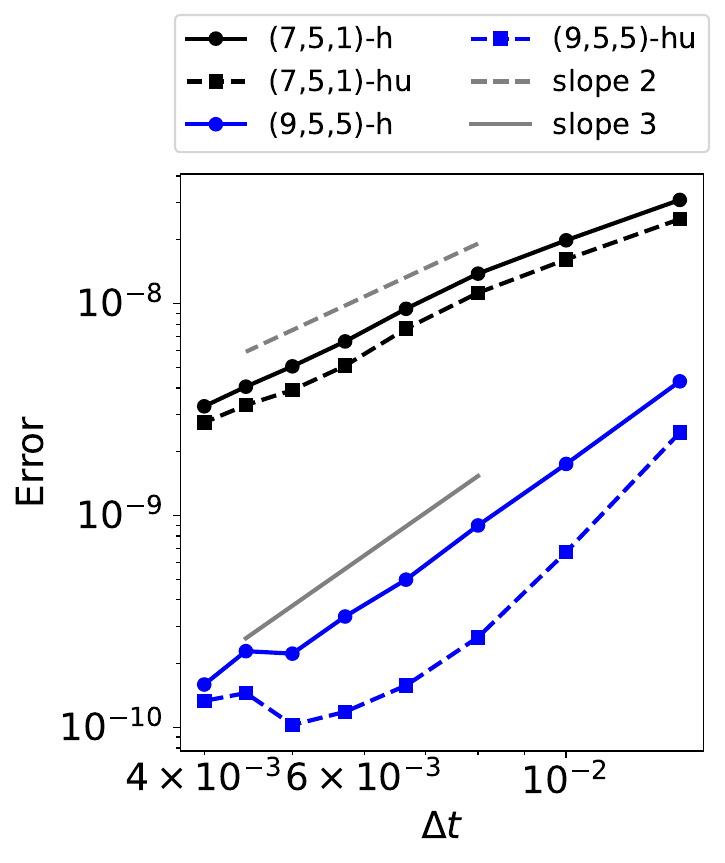}
	\end{minipage}
    \vspace{-.7em}
	\caption{Solution convergence for the 2D shallow-water equations at a final time $t=0.5$. Left: Third-order methods with different WSO. Middle: Fourth-order methods with different WSO. Right: Fifth-order methods with different WSO. The convergence graphs for the y-momentum ($hv$) are not displayed because they closely agree with the x-momentum ($hu$) for each order.}
	\label{fig:ErrConvg_ExactPolySol_2D_SWE}
\end{figure} 

We again use Clawpack with a fifth-order spatial discretization to solve this problem.
The error is evaluated at $t = 0.5$. The observed convergence, shown in Figure~\ref{fig:ErrConvg_ExactPolySol_2D_SWE}, is consistent with what one would expect based on the experiments above: the high WSO methods exhibit an order of accuracy of 3 (or even slightly better), while the traditional methods' accuracy is reduced to order 2. The high WSO methods yield dramatically smaller overall errors for the full range of step sizes considered.

\subsection{Comparison of WSO 3 methods}
Almost all ERK methods in the literature have weak stage order one, but there exist a few
methods with WSO two or three, most notably those of Skvortsov \cite{skvortsov2017avoid}.
Here we present a comparison between an \erk{} method of Skvortsov with WSO 3 (which we refer to as ERK313) and our (5,3,3) method, as well as a parallel iterated method of 9 stages with WSO 3 based on the construction in \cref{sec:high} (referred to as (9,3,3)).
The iterated method is determined by the choice of abscissas $\tilde{c}$; for this example we used the four-point Gauss quadrature nodes on $[0, 1]$.  We have observed similar results with methods based on other abscissas.
In Figure~\ref{fig:Comparison_WSO3_Mthds}, we present the convergence results for these three methods. The figure displays convergence plots for both the linear advection (left panel) and the inviscid Burgers' equation (right panel). We see that all three methods achieve 3rd-order convergence both problems. For the linear advection equation, all methods yield nearly identical errors, while for the inviscid Burgers' equation, our proposed (5,3,3) method exhibits the smallest errors.

\begin{figure}
    \centering
  \subfloat[Linear advection equation.]{\includegraphics[width=0.31\linewidth]{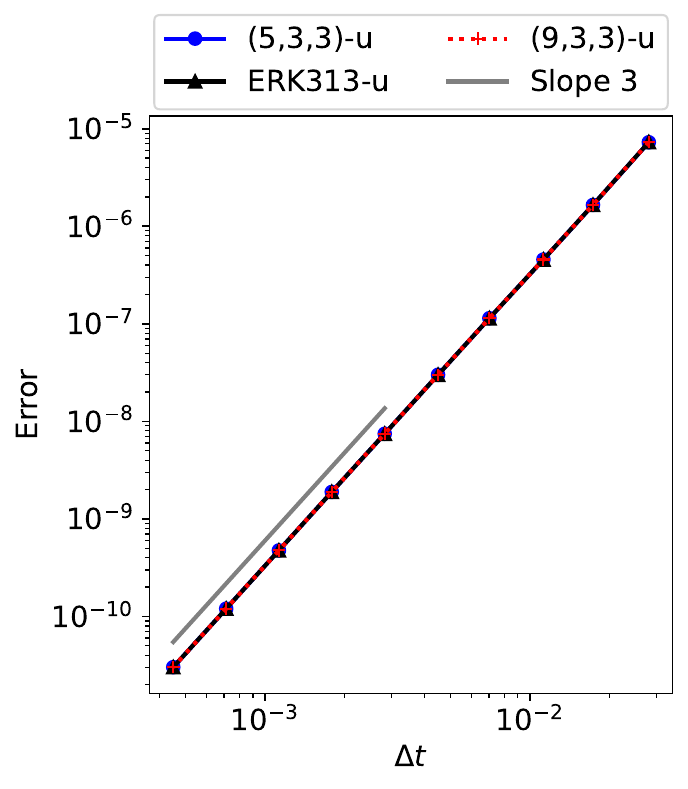} }
  \hfil
  \subfloat[Inviscid Burgers' equation.]{\includegraphics[width=0.31\linewidth]{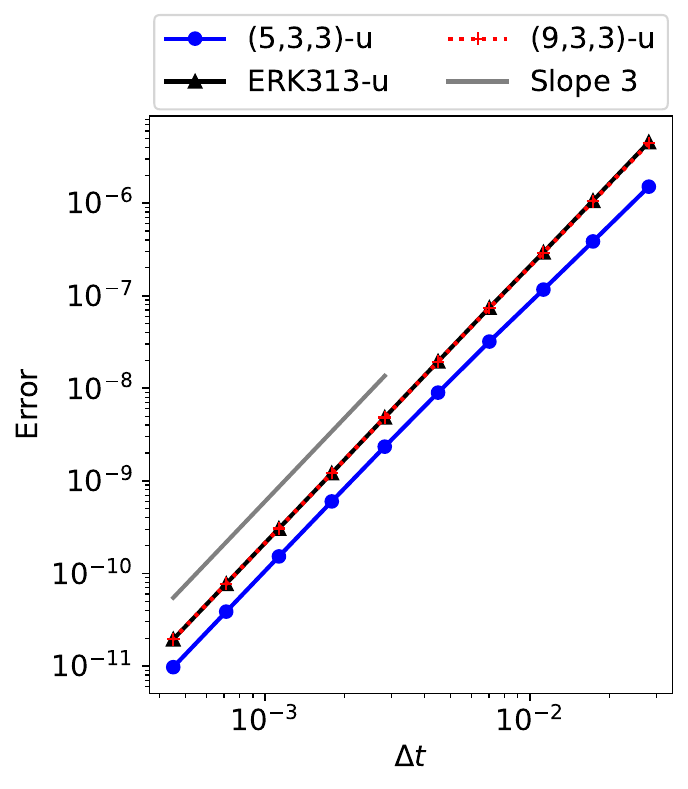} }
  \caption{Comparison of 3rd order \erk{} methods with WSO 3: (5,3,3) is our new method, ERK313 by Skvortsov, and (9,3,3) is the parallel iterated method with four-point Gauss quadrature abscissas on $[0,1]$ for $\Tilde{c}$.}
  \label{fig:Comparison_WSO3_Mthds}
\end{figure}

\section{Proofs of the main results}\label{Sec:PfsMainThm}
This section provides the proofs of Lemma~\ref{Lem:LBdimK}, Theorem~\ref{thm:qm1}, and Theorem~\ref{thm:qm2}, which are stated in \cref{subsec:mainresults}. 

\subsection{Proof of Lemma~\ref{Lem:LBdimK}}\label{subsec:lemLBdimK} 
Consider an arbitrary square matrix $\Xmat$; for each integer $m > 0$, we define the Vandermonde-like matrix
\begin{align}\label{Vm}
    V_m( \Xmat) & := 
    \lbrac\begin{array}{c|c|c|c} 
        \xsum & \xsum^2 & \cdots & \xsum^{m-1} 
    \end{array}\rbrac \, , \quad \textrm{where} \quad \xsum = X e \, .
\end{align}
We also define 
\begin{align}\label{Tm}
    \tau^{(k)}( \Xmat ) & := \Xmat \xsum^{k-1} - \frac{1}{k} \, \xsum^k  \, , \\
    T_m( \Xmat ) & := 
        \lbrac\begin{array}{c|c|c|c} 
            \tau^{(1)}( \Xmat ) & \tau^{(2)}( \Xmat ) & \cdots & \tau^{(m)}( \Xmat ) 
        \end{array}\rbrac \, .
\end{align}
Note that if $\Xmat=A$, then we have $\xsum=c$ and $\tau^{(j)}(\Xmat)=\tau^{(j)}$, the $j$th stage residual vector.
Given a matrix $M$, let $\Kspace(\Xmat,M)$ denote the smallest $\Xmat$-invariant space containing the column space of $M$, e.g., if $X$ is of size $r \times r$ then one has the explicit expression
\begin{align*}
    \Kspace(\Xmat,M) := \textrm{col} 
    \lbrac\begin{array}{c|c|c|c} 
        M & \Xmat \, M & \cdots & \Xmat^{r-1} \, M
    \end{array}\rbrac \, .
\end{align*}
With these definitions, $\Kspace(A,V_m(A) )$ is the span of all vectors $A^i c^j$ for
any integer $i$ and $1 \leq j \leq m-1$, and $K_m = \Kspace\left(A,T_m(A) \right)$.  

The next lemma establishes conditions under which $\Kspace( \Xmat, T_m( \Xmat) )$ contains the column space of $V_m(\Xmat)$. It provides a lower bound on the dimension of $K_q$, required in the proof of Lemma~\ref{Lem:LBdimK}.

\begin{lemma}[A matrix algebra result]\label{spaces-lemma}
Let $\Xmat \in \mathbb{R}^{r \times r}$ be a nilpotent matrix with row sums $\xsum = \Xmat e$. 
Suppose that $\xsum_j=0$ for some $j$ and let $m \geq 1$ denote the number of distinct entries of $\xsum$ (of which $m-1$ are non-zero). Then we have:
    \begin{enumerate}
        \item [(a)] $\Rank V_m(X) = m - 1$, 
        \item [(b)] $\xsum \in \Span\{ \xsum^2, \xsum^3, \dots, \xsum^{m}\}$,
        \item [(c)] $\Kspace\left( \Xmat, T_m( \Xmat)  \right) =  \Kspace\left( \Xmat, V_m( \Xmat ) \right)$, 
        \item [(d)] $\dim \Kspace\left( \Xmat, T_m( \Xmat)  \right) \geq m - 1$.
    \end{enumerate}
    Here $T_m(\Xmat)$ and $V_m(\Xmat)$ are defined in \eqref{Tm} and \eqref{Vm}.
\end{lemma}

\begin{remark}
    The assumption that $\xsum = \Xmat e$ in Lemma~\ref{spaces-lemma} is not necessary but is convenient in order to reduce the amount of notation required. \myremarkend
\end{remark}
\begin{remark} 
    Lemma~\ref{spaces-lemma} can be viewed as an extension of \cite[Thm.~3.11, Lem.~3.13]{biswas2022}, used in the theory of fully implicit and DIRK schemes, to \erk{}s. Here, the additional nilpotent structure of $\Xmat$ results in a tighter bound on $\dim K_m$. \myremarkend 
\end{remark}

\begin{proof}
    \noindent 
    (a) Due to the Vandermonde structure, $\Rank V_m(\Xmat)$ is the number of entries in $\xsum$ that are distinct from $0$, which is $m-1$. 
    
    \medskip
    \noindent
    (b) The columns of $V_m(\Xmat)$ are linearly independent and form a basis for the $(m-1)$-dimensional space $\Span\{\xsum, \xsum^2,\dots\}$. Write $\xsum^{m}_i = \sum_{\jind=1}^{m-1} \alpha_\jind \xsum^{\jind}_i$  for $1 \le i \le r \, ,$ and divide component-wise by $\xsum_i$:
    \begin{equation}\label{Eq:reduced}
        \xsum^{m-1}_i = \sum_{\jind=0}^{m-2} \alpha_{\jind + 1} \xsum^{\jind}_i\, ,
        \quad\quad 1 \le i \le r \, , \quad \xsum_i \neq 0 \, .
    \end{equation}        
    If $\alpha_1=0$, then \eqref{Eq:reduced} implies $\xsum^{m-1}$ is a linear combination of $\xsum^1, \dots, \xsum^{m-2}$, contradicting the linear independence of the columns of $V_m(\Xmat)$. Thus $\alpha_1\ne 0$ and we can write
    \begin{equation*}
        \xsum = \frac{1}{\alpha_1}\left(\xsum^{m} - \sum_{\jind=2}^{m-1} \alpha_{\jind} \xsum^{\jind}\right).
    \end{equation*}

    \smallskip
    \noindent
    (c) In general, to prove $\Kspace\left( \Xmat, M_1 \right) = \Kspace\left( \Xmat, M_2 \right)$ it suffices to show that $\Kspace\left( \Xmat, M_2 \right)$ contains the columns of $M_1$ and vice versa.  By part (b), $\Kspace\left( \Xmat,V_m(X) \right)$ contains all vectors of the form $\Xmat^i \, \xsum^j$ ($i \geq 0$ and $j \geq 1$) and thus all columns of $T_m(\Xmat)$.  It remains to show that $\Kspace\left( \Xmat, T_m(\Xmat) \right)$ contains the columns of $V_m(\Xmat)$, which we prove via induction. 

    We first show that if $\Kspace\left( \Xmat, T_m(\Xmat) \right)$ contains the columns of $\Xmat^i \, V_m(\Xmat)$ (for some $i \geq 1$), i.e.,
    \begin{align}\label{Eq:ih}
        \left\{ \Xmat^i \, \xsum, \Xmat^i \, \xsum^2, \dots, \Xmat^i \, \xsum^{m-1} \right\} 
        & \subseteq \Kspace\left( \Xmat, T_m(\Xmat) \right) \, ,
    \end{align}    
    then it contains the columns of $\Xmat^{i-1} \, V_m(\Xmat)$.  This result follows by combining the definition
    \begin{align*}
        -j \Xmat^{i-1} \, \tau^{(j)}(\Xmat) = \Xmat^{i-1} ( \xsum^j - j \Xmat \xsum^{j-1} ) & \in \Kspace\left( \Xmat, T_m(\Xmat) \right),
    & \mbox{for } 2 \le j \le m \, ,
    \end{align*}
    with the hypothesis \eqref{Eq:ih} to establish that $\Xmat^{i-1} \xsum^j \in \Kspace\left( \Xmat, T_m(\Xmat) \right)$ for $2 \le j \le m$.  Part (b) then implies that also $\Xmat^{i-1} \xsum \in \Kspace\left( \Xmat, T_m(\Xmat) \right)$, which concludes the induction step of the proof. 

    Starting the induction argument with $i = r$ for which the assumption holds trivially since $\Xmat^r=0$ ($\Xmat$ is nilpotent), and then going down to $i = 1$ shows that the columns of $V_m(\Xmat)$ are contained in $\Kspace\left( \Xmat, T_m(\Xmat) \right)$.
    
    \medskip
    \noindent
    (d) The result follows immediately from (c) since $\mathrm{col}\left(V_m(\Xmat) \right)$ is trivially an $(m-1)$-dimensional subspace of $\Kspace\left( \Xmat, V_m( \Xmat ) \right)$.
\end{proof}

Below, we apply Lemma~\ref{spaces-lemma} with $X$ given by a certain upper-left block of the Butcher matrix $A$, which we now identify. Given an \erk{} coefficient matrix $A\in \mathbb{R}^{s\times s}$ with row sums $c = Ae$ and WSO $q \geq 1$, let $\numc \le s$ denote the number of distinct entries of $c$, and $\hat{s}:=\hat{s}(q)$ the largest integer such that $\{c_1, c_2, \ldots, c_{\hat{s}}\}$ has $\mq$ distinct entries.  
We set $\mU$ the upper-left $\hat{s} \times \hat{s}$ block of $A$:
\begin{align}\label{Eq:Upperblock}
    A = \begin{\bracChoice}
        \mU & 0 \\
        * & *
    \end{\bracChoice} \, , \quad \textrm{where} \quad \mU \in \mathbb{R}^{\hat{s} \times \hat{s}} \, . 
\end{align}
Note that the definition implies that $c_{\hat{s}+1}$ differs from $c_1, \ldots, c_{\hat{s}}$. 

An important observation is that because $A$ is lower triangular, the space $\Kspace\left( \mU, T_q(\mU) \right) \subseteq \mathbb{R}^{\hat{s}}$ is the projection of $K_q$ onto the first $\hat{s}$ coordinate directions, because
\begin{align}\label{Eq:Projection}
    \lbrac\begin{array}{cc} 
        I  & 0
    \end{array}\rbrac   
    A^i \tau^{(j)} = \mU^i \, \tau^{(j)}(\mU) \in \Kspace\left( \mU, T_q(\mU) \right) .
\end{align}
Here $I$ is the $\hat{s} \times \hat{s}$ identity and $0$ the $\hat{s} \times (s-\hat{s})$ zero matrix.  The projection identity \eqref{Eq:Projection} implies 
\begin{align}\label{eq:dimProjbound}
    \dim \Kspace\left( \mU, T_q(\mU) \right) \leq \dim K_q \, .
\end{align}

\begin{proof} (of Lemma~\ref{Lem:LBdimK} from \cref{subsec:mainresults})  We apply Lemma~\ref{spaces-lemma} with $\Xmat = \mU$ ($r = \hat{s}$, $m = q$) defined by \eqref{Eq:Upperblock}.  If $q \ge \numc$, then $\Xmat = A$, $\xsum = c$, and Lemma~\ref{spaces-lemma}(c) implies $K_q=\Kspace\left(A,V_q(A) \right)$, so $c\in K_q$. The WSO conditions imply that $b^T c=0$, limiting $p = 1$.  If $q < \numc$, Lemma~\ref{spaces-lemma}(d) with \eqref{eq:dimProjbound} implies $\dim K_q \geq m - 1$.
\end{proof}

\subsection{Proof of Theorems~\ref{thm:qm1} and \ref{thm:qm2}}\label{subsec:proofqm}
We now turn our attention to proving Theorems~\ref{thm:qm1} and~\ref{thm:qm2}. Theorem~\ref{thm:qm1} is premised on S-reducibility, which is one of the two ways in which a \rk{} method is equivalent to, and may be replaced with, an \rk{} method with fewer stages.  We use the following criterion for S-reducibility, which is found in \cite[eq.~12.31, Section~IV]{wanner1996solving}
%
.

Given a (disjoint) partition of the integers $\cup_{i = 1}^{r} S_i = \{1, 2, \ldots, s\}$ for $r < s$, define the corresponding \emph{indicator vectors} for the sets $S_j$ as
\begin{align*}
    s_j = \sum_{i \in S_j} e_i \;, 
\end{align*}
where $e_i$ are the $i$-th unit vectors. Collectively the vectors $s_j$ are pairwise orthogonal, non-negative, and sum to $e$ (i.e., they form a discrete ``partition of unity''). Define the $s \times r$ \emph{partition matrix} $S = (s_{ij})_{1\leq i\leq s, 1\leq j \leq r}$ as
\begin{align}\label{Eq:partmat}   
    S = \lbrac\begin{array}{c|c|c|c} 
            s_1 & s_2 & \cdots & s_r 
        \end{array}\rbrac    
    \, , \quad \textrm{with entries} 
    \quad 
        s_{ij} = \begin{cases}
        1 & i \in S_j \\
        0 & \mbox{otherwise}\, .
    \end{cases}
\end{align}

\begin{definition}[Equation (12.31) of \cite{wanner1996solving}]\label{Def:altSred}
    A \rk{} method is $S$-reducible if and only if there is a non-trivial partition such that the associated partition matrix $S \in \mathbb{R}^{s \times r}$ ($r < s$) defined by \eqref{Eq:partmat} satisfies
    \begin{align}\label{Eq:AinvPart}
        A \, S = S \, \Bmat \, ,
    \end{align}
    for some matrix $\Bmat \in {\mathbb R}^{r \times r}$.  In other words, the method is S-reducible if and only if $\mathrm{col}(S)$ is $A$-invariant for a non-trivial partition. 
\end{definition} 

\begin{lemma} \label{lem:c-distinct}
Let a confluent \rk{} method be given and suppose that the column space of $V := V_{\numc}(A)$ is $A$-invariant.  Then the method is S-reducible.
\end{lemma}
\begin{proof}
We show that there exists a matrix $\tilde{\Bmat}$ satisfying $AS=S \tilde{\Bmat}$.  Partition the $s$ abscissa values $c_1, \ldots, c_s$ into sets $S_i$ ($i = 1, \ldots, \numc$) such that two abscissae belong to the same set if and only if they are equal.  Define the partition matrix $S$ as in \eqref{Eq:partmat}. Since the scheme is confluent the partition is non-trivial (i.e., $\numc < s$).  

Let $\hat{c}\in{\mathbb R}^{\numc}$ denote the vector remaining after duplicate entries are removed from $c$, and define
\begin{align} \label{V}
    \hat{V} & := 
        \lbrac\begin{array}{c|c|c} 
            \hat{c} & \cdots & \hat{c}^{\numc-1} 
        \end{array}\rbrac \, .
\end{align}
Hence 
$\begin{\bracChoice}
    e & \hat{V}   
\end{\bracChoice}
$
is a (square, invertible) Vandermonde matrix.  Observe that 
\begin{align}\label{BlockRelation}
    \begin{\bracChoice}
        e & V   
    \end{\bracChoice}
    = S
    \begin{\bracChoice}
        e & \hat{V}   
    \end{\bracChoice} \, , 
    \ \ \ \ \ \mbox{ so } \ \ \ \ 
    S = \begin{\bracChoice}
        e & V   
    \end{\bracChoice}    
    \begin{\bracChoice}
        e & \hat{V}   
    \end{\bracChoice}^{-1} \, .
\end{align}
In a slight abuse of notation, $e$ appears twice in \eqref{BlockRelation} with different dimensions.  We can then write (since $AV=VB$ for some matrix $B$, by the $A$-invariance of $V$)
\begin{align*}
    AS & = 
    \begin{\bracChoice}
        Ae & A V 
    \end{\bracChoice}
    \begin{\bracChoice}
        e & \hat{V}   
    \end{\bracChoice}^{-1} 
    = 
    \begin{\bracChoice}
        S \hat{c} & V B
    \end{\bracChoice}
     \begin{\bracChoice}
        e & \hat{V}   
    \end{\bracChoice}^{-1}
    = S \tilde{B} \, ,   
\end{align*}
where in the last equality we have used the fact that $V = S \hat{V}$ from \eqref{BlockRelation} to obtain
$ \tilde{\Bmat} = 
\begin{\bracChoice}
    \hat{c} & \hat{V} B
\end{\bracChoice}
\begin{\bracChoice}
    e & \hat{V}   
\end{\bracChoice}^{-1} $.
Hence the method is S-reducible.
\end{proof}

To prove the first part of Theorem~\ref{thm:qm1}, we need to generalize Lemma \ref{lem:c-distinct} so as to apply it to the upper-left block $\mU$ in place of $A$.

\begin{corollary}\label{cor-reducible} Let $\mU$ be defined as in \eqref{Eq:Upperblock}. If $\bar{u} = \mU e$ contains repeated entries, and the column space of $V_q(\mU)$ is $\mU$-invariant, then the \rk{} method $(A,b)$ (for any $b$) is S-reducible.    
\end{corollary}
\begin{proof}
    We can apply Lemma~\ref{lem:c-distinct} with $A$ replaced by $\mU$ and $V$ replaced by $V_q(\mU)$ and reach the same conclusion (i.e., partition the first $\hat{s}$ integers exactly as in the proof of Lemma~\ref{lem:c-distinct} with $A$ replaced by $\mU$, and take a trivial partition for the remaining $s - \hat{s}$ entries).
\end{proof}

\begin{proof} (of Theorem~\ref{thm:qm1} from \cref{subsec:mainresults}) The main idea is to show and exploit that 
\begin{equation}\label{eq:spaceEquals}
    \Kop\left( \mU ,V_q(\mU) \right) = \Col V_q( \mU ) \, . 
\end{equation}
This implies that the top components $K_q$ are the column space of a Vandermonde-like matrix. 

\medskip
\noindent 
(a) Apply Lemma~\ref{spaces-lemma} 
with $m = q$ to the matrix $\Xmat = \mU$ where $\mU$ is defined in \eqref{Eq:Upperblock}.  Thus, 
\begin{align}\label{Eq:Inequality}
    q - 1 = \mathrm{rank} \, V_q( \mU)  \le \dim \Kop(\mU,V_q(\mU))  \le \dim(K_q) = q-1 \, ,
\end{align}
where the rightmost equality arises from the hypothesis in the theorem. The inequality \eqref{Eq:Inequality} demands that \eqref{eq:spaceEquals} holds
and hence the columns of $V_q(\mU)$ are $\mU$-invariant. Corollary~\ref{cor-reducible} then requires the entries of $\hat{c} = \mU e$ (containing the first $q$ abscissas of $A$) to be distinct, otherwise the scheme is reducible.  By the definition of $\hat{s}$ in $\mU$ as the \emph{largest} integer such that $c_1, \ldots, c_{\hat{s}}$ has $q$ distinct entries, we also obtain that $c_{q+1}$ is distinct from all the preceding abscissas.

\medskip
\noindent
(b) From the proof of part (a), the matrix $\mU$ is of size $q \times q$, and from equation \eqref{Eq:Projection} the projection of the first $q$ components of $K_m$ is given by the space $\Kop\left( \mU ,V_q(\mU) \right) = \Col V_q(\mU)$. The matrix $V_q(\mU)$ has rank $q-1$ and a zero in the first row, hence 
$
\Col V_q(\mU) = \Col \, \begin{\bracChoice}
    0 & I    
\end{\bracChoice}^T
$
where $I$ is the $q-1$ identity matrix.  Thus, the $(q-1)$-dimensional space $K_q$ is equal to the column space of a matrix of the form
\begin{align}\label{Eq:ColSpaceKq}
    K_q = \Col\begin{\bracChoice} 0 \\ I \\ L \end{\bracChoice} \, \quad \textrm{for some} \quad L \in \mathbb{R}^{s-q \times q-1} \, .
\end{align}
We now show that \eqref{Eq:ColSpaceKq} implies $L$ satisfies both \eqref{Eq:FirstSylvester} and \eqref{Eq:SecondSylvester}.  The fact that $L$ solves \eqref{Eq:SecondSylvester} follows since $K_q$ is $A$-invariant: for every $x$ there exists $z$ such that
\begin{align*}
    A\begin{\bracChoice} 0 \\ I \\ L \end{\bracChoice} x & = \begin{\bracChoice} 0 \\ I \\ L \end{\bracChoice} z \, .
\end{align*}
In components the equation can be written as $z = A_{22} x$ and $A_{32} x + A_{33} L x = Lz$ where the blocks of $A$ are defined in \eqref{Eq:BlockA}. Eliminating $z$ and recognizing that $x$ is arbitrary yields \eqref{Eq:SecondSylvester}. 

The fact that $L$ also satisfies \eqref{Eq:FirstSylvester} follows by expressing $\tau^{(j)}$ in terms of the matrix defined in \eqref{Eq:ColSpaceKq}, as follows. There exists $m_j$ such that
\begin{align}\label{Eq:BlockTau}
    \tau^{(j)} = R
    \begin{\bracChoice}
        0 \\
        A_{22} c_U^{j-1} - \tfrac{1}{j} c_U^j \\
        A_{32} c_U^{j-1} + A_{33} c_L^{j-1} - \tfrac{1}{j} c_L^j 
    \end{\bracChoice} 
    = \begin{\bracChoice}
        0 \\
        I \\
        L
    \end{\bracChoice} m_j \, , \quad (2 \leq j \leq q) \, . 
\end{align}
Again, the block components of $A$ and $c$ in \eqref{Eq:BlockTau} are outlined in \eqref{Eq:BlockA}. 
The first non-zero block in \eqref{Eq:BlockTau} defines $m_j$; substituing the solution of $m_j$ into the second block of \eqref{Eq:BlockTau} (and collecting the terms involing $c_U^{j-1}$ ) yields 
\begin{align}\label{Eq:Step1FirstEq}
    A_{33} c_L^{j-1} - \tfrac{1}{j} c_L^j &=  \underbrace{\left( L A_{22} - A_{32} \right)}_{= A_{33} L \; \mathrm{by} \; \eqref{Eq:SecondSylvester} } c_U^{j-1} - \tfrac{1}{j} L c_U^j \, , \quad 
   (2 \leq j \leq q) \, .
\end{align}
Combining the $q-1$ vector equations defined by \eqref{Eq:Step1FirstEq} into a single matrix equation for $L$ and right-multiplying through by $V_{U}^{-1}$ yields \eqref{Eq:FirstSylvester}. Note that $V_{U}$ is invertible because the first $q$ entries of $c_j$ are distinct. 

Lastly, $L$ is unique: for every fixed $A$ (with $c_U$ distinct), \eqref{Eq:FirstSylvester} defines a uniquely solvable Sylvester equation for $L$ since the eigenvalues of $A_{33}$ and $W_{U} V_{U}^{-1}$ are disjoint (cf.~\cite[Theorem VII.2.1]{Bhatia1997}). Namely, the eigenvalues of $A_{33}$ are all zero, while the eigenvalues of $W_{U} V_{U}^{-1}$ are all non-zero since $W_{U} V_{U}^{-1}$ can be written as the product of invertible matrices, e.g., $W_{U} V_{U}^{-1} = C_U V_{q,U} D V_{q,U}^{-1}$, where $C_U = \textrm{diag}( c_2, c_3, \ldots, c_q )$ and $D= \textrm{diag}(\tfrac{1}{2}, \tfrac{1}{3}, \ldots , \tfrac{1}{q} )$. 

\medskip
\noindent
(c) From expression \eqref{Eq:ColSpaceKq}, the column space orthogonal to $K_q$ (which contains both $Y$ and thus $b$) is given by the following matrix: 
\begin{equation*}
        b \in Y \subseteq \Col \begin{\bracChoice}
            1 &  0 \\
            0 & -L^T \\
            0 &  I 
        \end{\bracChoice} \, .
\end{equation*}
    
\medskip
\noindent
(d) Since the last column of $LA_{22}$ and first row of $A_{33} L$ both vanish, this implies that the top-right entry of $A_{32}$ vanish for \eqref{Eq:SecondSylvester} to be solvable; i.e. $a_{q+1,q}=0$.
\end{proof}

The proof of Theorem~\ref{thm:qm2} makes use of the calculation in the proof of Theorem~\ref{thm:qm1}(c). 

\begin{proof}(of Theorem~\ref{thm:qm2} from \cref{subsec:mainresults})
    If $(A,b)$ satisfy Theorem~\ref{thm:qm1}(a), then $c_2, \ldots, c_q$ are non-zero and distinct, so $V_{U}$ is invertible and \eqref{Eq:FirstSylvester} well-defined.  
    If $A$ and $L$ satisfy the Sylvester equations \eqref{Eq:FirstSylvester}--\eqref{Eq:SecondSylvester}, then $W := \Col \begin{\bracChoice}
        0 & I & L^T 
    \end{\bracChoice}^T$ is (i)~$A$-invariant (by \eqref{Eq:SecondSylvester}), (ii)~contains $\tau^{(2)}, \ldots, \tau^{(q)}$ since equation \eqref{Eq:BlockTau} follows from satisfying \eqref{Eq:FirstSylvester}--\eqref{Eq:SecondSylvester} (e.g., working backwards via \eqref{Eq:Step1FirstEq}), and (iii)~is orthogonal to $b$ since $b$ has the form in \eqref{Eq:bvector}.  Because the vectors $\tau^{(j)}$ for $j = 1, \ldots, q$ lie in an $A-$invariant space orthogonal to $b$, the scheme has WSO of at least $q$.         
\end{proof}

\begin{remark}
    Schemes satisfying the sufficient conditions outlined by Theorem~\ref{thm:qm2} may have WSO greater than $q$ or may be reducible. \myremarkend
\end{remark}

\section{Conclusions and outlook} \label{sec:conclusion}
In this work  we have taken the concept of weak stage order (WSO) that was previously developed with respect to implicit methods, and applied it to \erk{} methods.
In so doing, we have developed a technique for overcoming order reduction when using \erk{} methods on problems with time-dependent boundary conditions, and provided optimized methods that achieve this goal.

We have established that \erk{} schemes can be constructed to have high WSO and classical order, and determined a lower bound on the number of stages required (Theorem~\ref{thm:qm1}); remarkably, this bound is sharp in almost all cases up to order 5 and WSO 5. Another crucial theoretical result is that there is no upper bound on the achievable order: any integer $p = q$ is achievable via a concrete \erk{} scheme, albeit one with a sizeable number of stages, $s = p^2$. These results bracket the number of stages needed for a scheme to achieve prescribed values of $p$ and $q$.

We have also shown that the WSO property has significant implications for numerical stability. Just as requiring classical order $p$ fixes the first $p+1$ coefficients of the stability function, so requiring WSO $q$ fixes (to zero) the highest $q-1$ coefficients.  For schemes with the minimum required number of stages ($s=p+q-1$), this completely determines the stability polynomial. We have also shown that methods with WSO $q>1$ cannot be strong stability preserving, since they must have at least one negative Butcher coefficient.

The construction of specific \erk{} schemes with high WSO is facilitated by their characterization
in terms of two Sylvester equations, given in Theorems~\ref{thm:qm2}--\ref{thm:qm3}.
This simultaneously ensures satisfaction of the WSO conditions and many of the classical
order conditions, and has enabled the construction of methods with exact (rational) coefficients
and small (optimal or near-optimal) error coefficients for $p, q \le 5$.
These methods can easily be used in existing application codes simply by replacing
the Butcher coefficients; this is much simpler than other remedies for order reduction, such as modified boundary conditions.  For example, we have already
implemented these methods in the Clawpack software as part of our numerical tests.

The new methods seem at a glance to have an increased cost due to the increased number of stages relative to traditional RK methods. However, the cost of these additional stages can be mostly avoided in many relevant situations (see \cref{sec:implementation}).  And even if not, based on our numerical tests the improvement in accuracy may more than offset the increased cost for applications where order reduction is a concern, but this is a topic that requires further study. Finally, the parallel iterated schemes of \cref{sec:high} are potentially as efficient as traditional methods if their potential for concurrency can be leveraged.

The foregoing discussion suggests that the use of a high WSO scheme is worthwhile whenever order reduction may be present in a numerical simulation.  But recognizing in advance the circumstances where order reduction will be observed is, in our experience, not a simple matter. The theory indicates that it may appear (among other situations) whenever time-dependent Dirichlet boundary conditions are imposed, but in our experience it does not always manifest as a dominant source of error. The terms responsible for order reduction, which depend on the equation, the semi-discretization, the initial and boundary data, and the mesh size, may in practical situations be dominated by other sources of error. This explains to some extent why order reduction is not always considered a practical concern, even in situations where it could theoretically arise.

An important limitation of the WSO theory is that it guarantees the avoidance of order reduction only for linear problems.  In tests on nonlinear problems, we have observed that methods with WSO 3 or higher exhibit third-order convergence, while low WSO methods show second-order convergence.  Thus high WSO schemes still yield a notable benefit on nonlinear problems.

Future work could include the development of improved schemes, by seeking schemes
with even smaller error coefficients, developing embedded error estimators or dense output
formulas, and finding methods with larger $p, q$.  The theory developed here
will be useful in these efforts.

\bibliographystyle{siamplain}
\bibliography{refs}

\appendix

\section{Fourth and fifth order methods}
\label{appendix:methods}

\subsection{(6,4,3) Method} \label{sec:ERK(6,4,3)}
\begin{equation*}
    \begin{butchertableau}{c|cccccc}
        0 & 0 & 0 & 0 & 0 & 0 & 0 \\
        1 & 1 & 0 & 0 & 0 & 0 & 0 \\
        \frac{1}{7} & \frac{461}{3920} & \frac{99}{3920} & 0 & 0 & 0 & 0 \\
        \frac{8}{11} & \frac{314}{605} & \frac{126}{605} & 0 & 0 & 0 & 0 \\
        \frac{5}{9} & \frac{13193}{197316} & \frac{39332}{443961} & \frac{86632}{190269} & -\frac{294151}{5327532} & 0 & 0 \\
        \frac{4}{5} & \frac{884721}{773750} & \frac{52291}{696375} & -\frac{155381744}{135793125} & -\frac{53297233}{355151250} & \frac{74881422}{85499375} & 0 \\ \hline
        \text{} & \frac{113}{2880} & \frac{7}{1296} & \frac{91238}{363285} & -\frac{1478741}{1321920} & \frac{147987}{194480} & \frac{77375}{72864}
    \end{butchertableau}
\end{equation*}

\subsection{(7,4,4) Method} \label{sec:ERK(7,4,4)}
\begin{small}
\begin{equation*}
    \begin{butchertableau}{rl|rl}
        a_{2,1} & \frac{13}{15} & a_{7,4} & \frac{615685898929080}{887641386333269} \\
        a_{3,1} & \frac{354503406167294455217584527356969321310499849}{679624939387359702842360408541392160411699600} & a_{7,5} & -\frac{88}{41} \\
        a_{3,2} & \frac{29553225679453489752042741666497760730650643}{2038874818162079108527081225624176481235098800} & a_{7,6} & \frac{63}{79} \\
        a_{4,1} & \frac{599677}{612720} & b_1 & -\frac{27983058641859756462867613}{8486495976646364788361250} \\
        a_{4,2} & \frac{1}{185} & b_2 & \frac{266859550993073190375211}{43133823812456533406250} \\
        a_{4,3} & \frac{1}{69} & b_3 & -\frac{3642903731392259905073408}{613543193666469780107625} \\
        a_{5,1} & \frac{11942118300581357822967470312387413892866711}{90616658584981293712314721138852288054893280} & b_4 & -\frac{59466320887669359732170224}{16752980798131655841946875} \\
        a_{5,2} & \frac{79816622789357424004900970571545142906303}{18123331716996258742462944227770457610978656} & b_5 & \frac{22530099787083474288594398}{3662271198716324657203125} \\
        a_{5,3} & \frac{10939005}{8358742409} & b_6 & \frac{13086932957294488}{71277904341826875} \\
        a_{5,4} & 0 & b_7 & \frac{12256178974}{9710853075} \\
        a_{6,1} & -\frac{2057331211140587771882165942948945576060485224020471}{5094460906663329618583273674295283629198217174096496} & c_1 & 0 \\
        a_{6,2} & \frac{37580055896186727391837634951840677945750522481251}{448734898514386546714588872865387677183262652640624} & c_2 & \frac{13}{15} \\
        a_{6,3} & -\frac{235459427251516205060}{1472801902839731775141} & c_3 & \frac{193}{360} \\
        a_{6,4} & -\frac{787608360}{15627214069} & c_4 & \frac{719}{720} \\
        a_{6,5} & \frac{24}{43} & c_5 & \frac{11}{80} \\
        a_{7,1} & \frac{793706393429237444430333112845341360638504851726921024780703}{806700576848993242482064062984309812448909584075544854292960} & c_6 & \frac{1}{36} \\
        a_{7,2} & -\frac{33849235109708152171969081938954415033838967121633968102863}{23685509164823635789628823956361427999363493832960729746080} & c_7 & \frac{193}{240} \\
        a_{7,3} & \frac{1821188984566562706805723220601}{956185881514873346828934914081}
    \end{butchertableau}
\end{equation*}
\end{small}

\subsection{(8,5,4) Method} \label{sec:ERK(8,5,4)}
\begin{small}
\begin{equation*}
\begin{butchertableau}{rl|rl|rl}
a_{2,1} & \frac{2}{31} & a_{7,3} & -\frac{271390788610093}{44561002480000} & b_7 & -\frac{42525800}{8688043} \\
a_{3,1} & \frac{8}{39} & a_{7,4} & \frac{16919854802127127}{33068912912100000} & b_8 & \frac{190120171223750}{63572266692433} \\
a_{3,2} & 0 & a_{7,5} & \frac{918241790299}{2569461804000} & c_1 & 0 \\
a_{4,1} & \frac{15}{38} & a_{7,6} & -\frac{1}{8} & c_2 & \frac{2}{31} \\
a_{4,2} & 0 & a_{8,1} & -\frac{69373518431251442108053395141546348749}{4382652560085449761027489727918400000} & c_3 & \frac{8}{39} \\
a_{4,3} & 0 & a_{8,2} & \frac{28436161533578442493717377903973791583}{1122666693846436666675352841982200000} & c_4 & \frac{15}{38} \\
a_{5,1} & \frac{23}{38} & a_{8,3} & -\frac{5846309065854115413909270194602947869}{606644216141135157002900448063680000} & c_5 & \frac{23}{38} \\
a_{5,2} & 0 & a_{8,4} & \frac{6129203519106929754603252009272053}{11862175903109203056563899370081250} & c_6 & \frac{31}{39} \\
a_{5,3} & 0 & a_{8,5} & \frac{242980026698914693640761833099573847}{314274501092549835332737438438856250} & c_7 & \frac{29}{31} \\
a_{5,4} & 0 & a_{8,6} & -\frac{38588365882306831}{818781973666952750} & c_8 & 1 \\
a_{6,1} & -\frac{281846119171}{64200240000} & a_{8,7} & -\frac{508578133539464}{4816364550982075} & \text{} & \text{} \\
a_{6,2} & \frac{289705767137}{45358567000} & b_1 & -\frac{13932812614910970806212030308137}{1494246680966212236480728656800} & \text{} & \text{} \\
a_{6,3} & -\frac{779567154093}{524247088000} & b_2 & \frac{442315248050515865700725458450027}{23731641831739396945145366137800} & \text{} & \text{} \\
a_{6,4} & \frac{199824989}{614863125} & b_3 & -\frac{21619621692735791984774655801338457}{1572963107476970769686133552792800} & \text{} & \text{} \\
a_{6,5} & -\frac{1}{25} & b_4 & \frac{4931046639398139760440943293895907}{887688100270302681290608525794300} & \text{} & \text{} \\
a_{7,1} & -\frac{5647052528401825871}{514607937760800000} & b_5 & -\frac{808732636620048337464280245511529}{1567883987541272156723519232078580} & \text{} & \text{} \\
a_{7,2} & \frac{80442150849469599005477}{4661884215626994720000} & b_6 & \frac{52162695}{22722574} & \text{} & \text{} \\
\end{butchertableau}
\end{equation*}
\end{small}

\subsection{(9,5,5) Method} \label{sec:ERK(9,5,5)}
\begin{small}
\begin{equation*}
    \begin{butchertableau}{rl|rl|rl|rl}
        a_{2,1} & \frac{1}{19} & a_{6,5} & 0 & a_{9,1} & \frac{544015925591990906117739018863}{21097279127167116142731264000} & b_7 & -\frac{13778944}{1751475} \\
        a_{3,1} & \frac{1}{6} & a_{7,1} & \frac{11448031}{2850816} & a_{9,2} & -\frac{51819957177912933732533469147783191}{1292529408768612025127952939417600} & b_8 & \frac{92889088}{11941875} \\
        a_{3,2} & 0 & a_{7,2} & -\frac{67411795275}{16590798848} & a_{9,3} & \frac{15141148893501140337719772533}{769541606770966638202880000} & b_9 & -\frac{714103988224}{149255126145} \\
        a_{4,1} & \frac{5}{16} & a_{7,3} & \frac{51073011}{43237376} & a_{9,4} & -\frac{22062343808701233885761491}{5740046662014404900523000} & c_1 & 0 \\
        a_{4,2} & 0 & a_{7,4} & -\frac{23353}{64148} & a_{9,5} & -\frac{180818957612953115541011736739}{146721986657116762265358336000} & c_2 & \frac{1}{19} \\
        a_{4,3} & 0 & a_{7,5} & \frac{583825}{8077312} & a_{9,6} & \frac{18393837528018836258241002593}{22366927394951953576613895000} & c_3 & \frac{1}{6} \\
        a_{5,1} & \frac{1}{2} & a_{7,6} & -\frac{1}{116} & a_{9,7} & -\frac{14372715851}{701966192290} & c_4 & \frac{5}{16} \\
        a_{5,2} & 0 & a_{8,1} & \frac{30521441823091}{1986340257792} & a_{9,8} & -\frac{3316780581}{34682124125} & c_5 & \frac{1}{2} \\
        a_{5,3} & 0 & a_{8,2} & -\frac{745932230071621375}{35792226257928192} & b_1 & \frac{201919428075343316424206867}{7205146638186855485778750} & c_6 & \frac{11}{16} \\
        a_{5,4} & 0 & a_{8,3} & \frac{42324456085}{5966757888} & b_2 & -\frac{979811820279525173317561445351}{23232888464237446713644747250} & c_7 & \frac{5}{6} \\
        a_{6,1} & \frac{11}{16} & a_{8,4} & \frac{775674925}{6453417096} & b_3 & -\frac{659616477161155066954978}{262813990730721440278125} & c_8 & \frac{16}{17} \\
        a_{6,2} & 0 & a_{8,5} & -\frac{38065236125}{28020473856} & b_4 & \frac{10343523856053877739219144704}{232857239079584284108576875} & c_9 & 1 \\
        a_{6,3} & 0 & a_{8,6} & \frac{18388001255}{24775053336} & b_5 & -\frac{2224588357354685208355760476}{50108519801935858605643125} & \text{} & \text{} \\
        a_{6,4} & 0 & a_{8,7} & -\frac{25}{138} & b_6 & \frac{704220346724742597999572733952}{31288349276326419946994221875} & \text{} & \text{}
    \end{butchertableau}
\end{equation*}
\end{small}

\end{document}